\documentclass[12pt,a4paper]{article}

\usepackage{amsmath, amsthm, amsfonts, amssymb,color,geometry,enumerate}
\usepackage{graphicx}
\usepackage{fancybox}
\usepackage{cases}
\usepackage{fancyhdr}

\theoremstyle{definition}

\newtheorem{theorem}{Theorem}[section]
\newtheorem{proposition}[theorem]{Proposition}
\newtheorem{lemma}[theorem]{Lemma}
\newtheorem{definition}[theorem]{Definition}
\newtheorem{remark}[theorem]{Remark}

\newtheorem{conjecture}[theorem]{Conjecture}
\newtheorem{corollary}[theorem]{Corollary}

\makeatletter

\@addtoreset{equation}{section}
\makeatother

\renewcommand{\epsilon}{\varepsilon}    

\begin{document}

\title{Limit theorems for classical, freely and Boolean max-infinitely divisible distributions}
\author{\Large{Yuki Ueda}}
\date{}

\maketitle

\abstract{We investigate a Belinschi-Nica type semigroup for free and Boolean max-convolutions. We prove that this semigroup at time one connects limit theorems for freely and Boolean max-infinitely divisible distributions. Moreover, we also construct a max-analogue of Boolean-classical Bercovici-Pata bijection, establishing the equivalence of limit theorems for Boolean and classical max-infinitely divisible distributions.}
\\
\vspace{1mm}



\section{Introduction}
\label{intro}
Denote by $\mathcal{P}$ and $\mathcal{P}^+$ the sets of all probability measures on $\mathbb{R}$ and $[0,\infty)$, respectively. We write $\Delta$ as the set of all distribution functions on $\mathbb{R}$. Moreover, we define the set $\Delta_+:=\{F\in \Delta: F(x)=0 \text{ for } x<0\}$.

A probability measure $\mu$ on $\mathbb{R}$ is said to be {\it $\circ$-infinitely divisible} if for any $n\in\mathbb{N}$ there exists $\mu_n\in\mathcal{P}$ such that
\begin{align*}
\mu=\overbrace{\mu_n\circ\cdots \circ \mu_n}^{n \text{ times}},
\end{align*}
where $\circ\in \{\ast, \boxplus, \uplus\}$. The operation $\ast$ is {\it classical convolution}, $\boxplus$ is {\it free convolution} and $\uplus$ is {\it Boolean convolution}. Let ${\bf ID}(\circ)$ be the set of all $\circ$-infinitely divisible distributions on $\mathbb{R}$, where $\circ\in\{\ast,\boxplus,\uplus\}$. Speicher and Woroudi \cite[Theorem 3.6]{SW} proved that ${\bf ID}(\uplus)=\mathcal{P}$. 

In \cite[Theorem 6.3]{BP99}, Bercovici and Pata obtained innovated results for three types of infinitely divisible distributions. For any sequences $\{\mu_n\}_n$ in $\mathcal{P}$ and any sequences $\{k_n\}_n$ of positive integers with $k_1<k_2<\cdots$, the following conditions are equivalent.
\begin{enumerate}[(1)]
\item there exists $\mu\in {\bf ID}(\ast)$  such that $\overbrace{\mu_n \ast \cdots \ast \mu_n }^{k_n \text{ times}}\xrightarrow{w} \mu$ as $n\rightarrow \infty$;
\item there exists $\nu \in {\bf ID}(\boxplus)$  such that $\overbrace{\mu_n \boxplus \cdots \boxplus \mu_n}^{k_n \text{ times}} \xrightarrow{w} \nu$ as $n\rightarrow\infty$;
\item there exists $\lambda\in {\bf ID}(\uplus)=\mathcal{P}$  such that $\overbrace{\mu_n \uplus \cdots \uplus \mu_n}^{k_n \text{ times}} \xrightarrow{w} \lambda$ as $n\rightarrow\infty$.
\end{enumerate}

According to this result, we can construct the {\it  Bercovici-Pata bijection} $\mathbf{B}_{\bf \ast\mapsto \boxplus}: {\bf ID}(\ast)\rightarrow {\bf ID}(\boxplus)$, $\mu\mapsto \nu$ and the {\it Boolean-free Bercovici-Pata bijection} $\mathbf{B}_{\bf \uplus\mapsto \boxplus}: {\bf ID}(\uplus)\rightarrow {\bf ID}(\boxplus)$, $\lambda\mapsto \nu$. These bijections play an important role to understand limit theorems (containing the central limit theorem) in three probability theories. For example, $\mathbf{B}_{\bf \ast\mapsto \boxplus}$ maps the normal distribution to the Wigner's semicircle law and $\mathbf{B}_{\bf \uplus\mapsto \boxplus}$ maps the symmetric Bernoulli distribution to the Wigner's semicircle law. In general, it is known that $\mathbf{B}_{\bf \ast\mapsto \boxplus}$ maps stable laws to free stable laws, and $\mathbf{B}_{\bf \uplus \mapsto \boxplus}$ sends Boolean stable laws to free stable laws (see \cite{BP99}, \cite{BN08}). 



Belinschi and Nica \cite[Theorem 1.1]{BN08} defined the important map to understand free infinite divisibility for probability measures and the Boolean-free Bercovici-Pata bijection as follows:
\begin{align*}
\mathbf{B}_t (\mu):=(\mu^{\boxplus (1+t)})^{\uplus\frac{1}{1+t}}, \qquad t\ge0, \hspace{2mm} \mu\in\mathcal{P}.
\end{align*}
It is known that $\mathbf{B}_t\circ\mathbf{B}_s=\mathbf{B}_{t+s}$ for every $t,s\ge0$. Hence the family $\{\mathbf{B}_t\}_{t\ge0}$ is a semigroup with respect to the composition of maps and it is called {\it Belinschi-Nica semigroup}. In \cite[Corollary 1.3]{BN08}, Belinschi and Nica obtained that $\mathbf{B}_t(\mu) \in {\bf ID}(\boxplus)$ for every $t\ge 1$ and $\mu\in\mathcal{P}$. Moreover, the map $\mathbf{B}_t$ is a homomorphism on $\mathcal{P}^+$ with respect to free multiplicative convolution $\boxtimes$ (the details of $\boxtimes$ were given by \cite{BV93}), that is, $\mathbf{B}_t (\mu\boxtimes\nu)=\mathbf{B}_t (\mu)\boxtimes \mathbf{B}_t (\nu)$ for all $\mu,\nu\in\mathcal{P}^+$ (see \cite[Theorem 1.1]{BN08}).  Finally, Belinschi and Nica clarified the Boolean-free Bercovici-Pata bijection by using the above semigroup at time one.

\begin{theorem}\label{BN}
(See \cite[Theorem 1.2]{BN08}) The Boolean-free Bercovici-Pata bijection coincides with the map $\mathbf{B}_1$.
\end{theorem}

Max-probability (extreme value) theory is concerned with maxima of real random variables. Define $X\lor Y:=\max\{X,Y\}$ as the maximum of real random variables $X$ and $Y$. Denote by $F_X$ the cumulative distribution function of $X$, that is, $F_X(x):=\mathbb{P}(X\le x)$. If $X$ and $Y$ are independent random variables, then 
\begin{align*}
F_{X\lor Y}(x)&=\mathbb{P}(X\lor Y\le x)=\mathbb{P}(X\le x, Y\le x)\\
&=\mathbb{P}(X\le x)\mathbb{P}(Y\le x)=F_X(x)F_Y(x), \qquad x\in\mathbb{R}.
\end{align*}
According to the above equation, we define {\it max-convolution} $\lor$ of distribution functions $F$ and $G$ by setting $F\lor G:=FG$. The concept of max-infinite divisibility acts like classical convolution case, that is, $F$ is said to be {\it max-infinitely divisible} if for any $n\in\mathbb{N}$ there exists a distribution function $F_n$ such that $F=\overbrace{F_n\lor \cdots \lor F_n}^{n \text{ times}}=F_n^n$. However, every distribution function $F$ is max-infinitely divisible since the $n$-th root $F_n:=F^{1/n}$ is also distribution function. A non-degenerate distribution function $F$ is called {\it max-stable law} if 
\begin{align*}
F_{\frac{(X_1\lor \cdots \lor X_n)-b_n}{a_n}}\xrightarrow{w} F, \qquad n\rightarrow\infty,
\end{align*}
for some sequence $\{X_k\}_k$ of independent identically distributed $\mathbb{R}$-valued random variables, a sequence $\{a_k\}_k$  of positive real numbers and a sequence $\{b_k\}_k$ of real numbers.  Fisher and Tippett, Fr\'{e}chet and Gnedenko  proved that every max-stable law is an {\it extreme value distribution} which is one of following types:
\begin{itemize}
\item ${\text F_{\text I}^{\text{classical}}}(x)= \exp(-\exp(-x))$, \hspace{1mm} $x\in\mathbb{R}$: {\it Gumbel distribution};
\item ${\text F_{\text{II}}^{\text{classical}}}(x)= \begin{cases}
\exp(-x^{-\alpha}), & x>0 \\
0, & x\le 0,
\end{cases}$ \hspace{1mm} $\alpha>0$: {\it Fr\'{e}chet distribution};
\item ${\text F_{\text{III}}^{\text{classical}}}(x)=\begin{cases}
1, & x> 0\\
\exp(-(-x)^{\alpha}),& x\le 0,
\end{cases}$ \hspace{1mm} $\alpha>0$: {\it Weibull distribution}.
\end{itemize}
These distribution functions are the most important ones in extreme value theory, and they are used in max-value statistics.

In noncommutative probability theory, we can also establish extreme value theory as follows. For noncommutative real random variables (selfadjoint operators) $X$ and $Y$, we denote by $X\lor Y$ the maximum of $X$ and $Y$ with respect to the spectral order (which was introduced by \cite{O71}). In this paper, we call $X\lor Y$ the {\it spectral maximum} of $X$ and $Y$, while considering $X$ and $Y$ as noncommutative real random variables. Note the slight abuse of notation resulting from the dual use of the symbol $\lor$. By using the spectral order, we can establish the concepts of max-convolution, max-infinitely divisible distributions and max-stable laws (extreme value distributions) in free and Boolean settings (see Sections 3 and 4, for details).

In Section 5, we obtain the limit theorem for freely max-infinitely divisible distributions from Boolean max-limit theorem. Before our results, we give important concepts in our paper, namely, free max-convolution, Boolean max-convolution and max-Belinschi-Nica semigroup. We denote by $F \Box \hspace{-.97em}\lor G$ the (spectral) distribution function of the spectral maximum $X\lor Y$ of freely independent (noncommutative) real random variables $X$ and $Y$ whose (spectral) distribution functions are $F$ and $G$, respectively. Here, the (spectral) distribution functions of noncommutative random variables are defined in Section 2. The operation $\Box \hspace{-.72em}\lor$ is called {\it free max-convolution} (see Section 3.1, for details). We denote by $F \cup \hspace{-.85em}\lor G$ the (spectral) distribution function of the spectral maximum $X\lor Y$ of Boolean independent (noncommutative) positive random variables $X$ and $Y$ whose (spectral) distribution functions are $F$ and $G$, respectively. The operation $\cup \hspace{-.65em}\lor$ is called {\it Boolean max-convolution} (see Section 4.1, for details). A map $\mathbf{B}_t^M:\Delta\rightarrow \Delta$, is defined by $\mathbf{B}_t^M(F):=(F^{\Box \hspace{-.55em}\lor 1+t})^{\cup \hspace{-.55em}\lor \frac{1}{1+t}}$ for $t\ge 0$. The reader can know more precise definitions of $F^{\Box \hspace{-.55em}\lor t}$ and $F^{\cup \hspace{-.55em}\lor t}$ in Sections 3 and 4, respectively. In Section 5, we show that the family $\{\mathbf{B}_t^M\}_{t\ge0}$ makes a semigroup with respect to the composition of maps. We call the family $\{\mathbf{B}_t^M\}_{t\ge0}$ {\it max-Belinschi-Nica semigroup}. 
Then the following theorem is obtained.

\begin{theorem}\label{thm1}
Let $\{F_n\}_n$ be a sequence in $\Delta_+$ and $\{k_n\}_n$ a sequence of positive integers such that $k_1<k_2<\cdots$. If there exists $F\in \Delta_+$ such that $\overbrace{F_n \cup \hspace{-.88em}\lor\cdots \cup \hspace{-.88em}\lor F_n}^{k_n \text{ times}}\xrightarrow{w} F$, then $\overbrace{F_n\Box \hspace{-.95em}\lor \cdots \Box \hspace{-.95em}\lor F_n}^{k_n \text{ times}}\xrightarrow{w} \mathbf{B}_1^M(F)$ as $n\rightarrow\infty$.
\end{theorem}

In Section 6, we prove the equivalence of limit theorems for classical and Boolean max-infinitely divisible distributions.

\begin{theorem}\label{thm1.3}
Consider a sequence $\{F_n\}_n$ in $\Delta_+$ and a sequence $\{k_n\}_n$ of positive integers with $k_1<k_2<\cdots$. The following conditions are equivalent.\\
(1) There exists $F\in\Delta_+$ such that $\overbrace{F_n\lor \cdots \lor F_n}^{k_n \text{ times}}\xrightarrow{w} F$ as $n\rightarrow \infty$;\\
(2) There exists $G\in\Delta_+$ such that $\overbrace{F_n\cup \hspace{-.88em}\lor \cdots \cup \hspace{-.88em}\lor F_n}^{k_n \text{ times}}\xrightarrow{w} G$ as $n\rightarrow\infty$.
\end{theorem}
Thus we get the bijection $F\mapsto G$, where $F$ and $G$ are distributions in Theorem \ref{thm1.3}. It is called the {\it Boolean-classical max-Bercovici-Pata bijection}. Moreover, we find an explicit representation of the bijection.

\section{Spectral order and distribution functions of noncommutative random variables}
\label{sec:1}
Let $(\mathcal{M},\tau)$ be a tracial $W^*$-probability space, that is, $\mathcal{M}$ is a von Neumann algebra and $\tau$ is an ultraweakly continuous faithful tracial state. We may assume that $\mathcal{M}$ acts on a Hilbert space $\mathcal{H}$ (e.g. taking a Hilbert space $\mathcal{H}=L^2(\mathcal{M},\tau)$ with an inner product defined by $\langle X,Y\rangle_\mathcal{H}:=\tau(Y^*X)$ for all $X,Y\in\mathcal{M}$). We denote by $\mathcal{M}_{sa}$ and $\mathcal{P}(\mathcal{M})$ the set of all selfadjoint operators and projections in $\mathcal{M}$.

Firstly, we define an order $\le$ on $\mathcal{M}_{sa}$ as follows:
\begin{align*}
X\le Y \text{ for } X,Y \in \mathcal{M}_{sa} \Leftrightarrow Y-X \text{ is a positive operator in }\mathcal{M}.
\end{align*}
The order $\le$ is called the {\it operator order}. Note that $\mathcal{P}(\mathcal{M})$ is a complete lattice with respect to the operator order. For any $P,Q\in \mathcal{P}(\mathcal{M})$, we define the following projections on $\mathcal{H}$:
\begin{align*}
(P\lor Q)\mathcal{H}&:=cl(P\mathcal{H}+Q\mathcal{H});\\
(P\land Q)\mathcal{H}&:=P\mathcal{H}\cap Q\mathcal{H}.
\end{align*}
In fact, $P\lor Q$ and $P\land Q$ are in $\mathcal{P}(\mathcal{M})$. Moreover the projections $P\lor Q$ and $P\land Q$ are the maximum and the minimum of $P$ and $Q$ with respect to the operator order, respectively.

Next, we discuss the maximum and the minimum of selfadjoint operators in $\mathcal{M}$. However, the set $\mathcal{B}(\mathcal{H})_{sa}$ of all (bounded) selfadjoint operators on a Hilbert space $\mathcal{H}$ does not form a complete lattice with respect to the operator order (see \cite{K51}). From the reason, we need to find an another order on $\mathcal{M}_{sa}$ (in particular, $\mathcal{B}(\mathcal{H})_{sa}$) to consider the maximum of selfadjoint operators. In \cite{O71}, Olson introduced the spectral order as a new order on the set $\mathcal{B}(\mathcal{H})_{sa}$. After that, Ben Arous and Voiculescu in \cite{AV06} extended this order to $\mathcal{M}_{sa}$. For any $X\in\mathcal{M}_{sa}$ and any Borel set $T$ in $\mathbb{R}$, an operator $\mathbf{E}(X; T)\in \mathcal{P}(\mathcal{M})$ denotes the corresponding spectral projection. The {\it spectral order} on $\mathcal{M}_{sa}$ is defined by $X\prec Y \Leftrightarrow \mathbf{E}(X;[t,\infty))\le \mathbf{E}(Y;[t,\infty))$ for all $X,Y\in\mathcal{M}_{sa}$ and $t\in\mathbb{R}$.  It is known that the spectral order $\prec$ is extended to unbounded selfadjoint operators affiliated with a $W^*$-probability space $(\mathcal{M},\tau)$. In \cite{Ando} ($\mathcal{M}$: matrix algebra) and \cite{AV06}, for any $X,Y\in\mathcal{M}_{sa}$, they defined operators $X\lor Y$ and $X\land Y$ on $\mathcal{H}$ as follows:
\begin{align*}
\mathbf{E}(X\lor Y; (t,\infty)))&:=\mathbf{E}(X; (t,\infty))\lor \mathbf{E}(Y; (t,\infty)), \qquad t\in\mathbb{R};\\
\mathbf{E}(X\land Y; [t,\infty))&:=\mathbf{E}(X; [t,\infty))\land \mathbf{E}(Y; [t,\infty)),\qquad t\in\mathbb{R}.
\end{align*}
Note that $X\lor Y$ and $X\land Y$ are in $\mathcal{M}_{sa}$. By definition, $X\lor Y$ and $X\land Y$ are the maximum and the minimum of $X,Y\in\mathcal{M}_{sa}$ with respect to the spectral order. It follows from \cite{AV06} that
\begin{align*}
\mathbf{E}(X\lor Y; [t,\infty))&=\mathbf{E}(X; [t,\infty))\lor \mathbf{E}(Y; [t,\infty)), \qquad t\in\mathbb{R},
\end{align*}
since $\tau$ is tracial. Moreover, the definitions and properties of $\lor$ and $\land$ extend to unbounded selfadjoint operators affiliated with $(\mathcal{M},\tau)$.

For an (unbounded) selfadjoint operator $X$ affiliated with $(\mathcal{M},\tau)$, we define a function $F_X$ by setting $F_X(x):=\tau(\mathbf{E}(X;(-\infty,x]))$ for all $x\in\mathbb{R}$. It follows from the Borel functional calculus that $F_X$ is a distribution function of some $\mu_X\in\mathcal{P}$, that is, $\mu_X((-\infty,x])=F_X(x)$ for all $x\in\mathbb{R}$. In this paper, we call $F_X\in \Delta$ a {\it spectral distribution of $X$ with respect to $\tau$} while considering noncommutative probability theory.  

\section{Theory of freely max-infinitely divisible distributions}
\label{sec:2}

\subsection{Free max-convolution}
\label{sec:3}
If $X$ and $Y$ are freely independent real random variables (selfadjoint operators) affiliated with a tracial $W^*$-probability space $(\mathcal{M},\tau)$, then we have
\begin{align*}
F_{X\lor Y}(x)=(F_X(x)+F_Y(x)-1)_+, \qquad x\in\mathbb{R},
\end{align*}
where $(a)_+:=\max\{a,0\}$ for $a\in\mathbb{R}$ (see \cite[Corollary 3.3]{AV06} for details). Considering the fact, we can define free max-convolution as follows.

\begin{definition}
We define an operation $\Box \hspace{-.72em}\lor$ on $[0,1]$ by setting $u\Box \hspace{-.95em}\lor v:=\max\{u+v-1,0\}$. For $F,G\in \Delta$, we define a distribution function $F \Box \hspace{-.95em}\lor G$ as 
\begin{align*}
(F \Box \hspace{-.95em}\lor G )(x):= F(x)\Box \hspace{-.95em}\lor  G(x), \qquad x\in \mathbb{R}.
\end{align*}
The operation $\Box \hspace{-.73em}\lor: \Delta\times \Delta\rightarrow \Delta$, $(F,G)\mapsto F\Box \hspace{-.95em}\lor G$, is called {\it free max-convolution}. Note the slight abuse of notation resulting from the dual use of the symbol $\Box \hspace{-.70em}\lor$.
\end{definition}

By definition of free max-convolution, we obtain
\begin{align*}
F^{\Box \hspace{-.55em}\lor n}:=\overbrace{F \Box \hspace{-.95em}\lor\cdots \Box \hspace{-.95em}\lor F}^{n \text{ times}}=(nF-(n-1))_+, \qquad n\in\mathbb{N}, \hspace{2mm} F\in\Delta.
\end{align*}

For $F\in \Delta$, we define an interval $[\alpha(F),\omega(F)]\subset[-\infty,\infty]$ by setting
\begin{align*}
\alpha(F)&:=\sup \{x\in\mathbb{R}: F(x)=0\};\\
\omega(F)&:=\inf\{x\in \mathbb{R}: F(x)=1\}.
\end{align*}

For any positive integers $n\ge2$, we have $\alpha(F^{\Box \hspace{-.60em}\lor n})>-\infty$,
\begin{align*}
\alpha(F^{\Box \hspace{-.55em}\lor n})=\sup\{x\in\mathbb{R}: F(x)\le 1-1/n \},
\end{align*}
and $\lim_{n\rightarrow \infty} \alpha(F^{\Box \hspace{-.55em}\lor n})=\omega (F)$. If $F$ is continuous on $\mathbb{R}$, then there is $u\in\mathbb{R}$ such that $F(u)=1-1/n$, that is, $\alpha(F^{\Box \hspace{-.55em}\lor n})=\sup\{u: F(u)=1-1/n\}$ by applying the intermediate value theorem.

For $u\in [0,1]$ and $t\in \mathbb{R}$, we define $u^{\Box \hspace{-.55em}\lor t}$ as follows:
\begin{align*}
u^{\Box \hspace{-.55em}\lor t}:=(tu-(t-1))_+, \qquad t\in \mathbb{R}.
\end{align*}
In \cite[Lemma 6.9]{AV06}, Ben Arous and Voiculescu extended the $n$-fold free max-convolution $F^{\Box \hspace{-.55em}\lor n}$ of a distribution function $F\in\Delta$ to a function $F^{\Box \hspace{-.55em}\lor t}$ on $\mathbb{R}$ for $t\ge 1$ as follows:
\begin{align*}
F^{\Box \hspace{-.55em}\lor t}(x):=F(x)^{\Box \hspace{-.55em}\lor t}, \qquad x\in\mathbb{R}, \hspace{2mm} t\ge1.
\end{align*}
This is a distribution function with $\alpha(F^{\Box \hspace{-.55em}\lor t})=\sup\{x\in\mathbb{R}: F(x)\le 1- 1/t\}$ and the map $t\mapsto F^{\Box \hspace{-.55em} \lor t}$ is weakly continuous on $[1,\infty)$. Note that one does not get a distribution function when $t<1$. 

Let $\Lambda^\lor$ be a map on $[0,1]$ defined by
\begin{align*}
\Lambda^\lor(u):&=(1+\log u)_+, \qquad u\in (0,1],\\
\Lambda^\lor(0):&=0.
\end{align*}
It is known that $\Lambda^\lor$ is a homomorphsim between the semigroups $([0,1],\cdot)$ and $([0,1],\Box\hspace{-.70em}\lor)$, that is $\Lambda^\lor (uv)=\Lambda^\lor (u) \Box\hspace{-.95em}\lor \Lambda^\lor (v)$ for all $u,v\in [0,1]$. By definition, note that $\Lambda^\lor (u^t)=\Lambda^\lor (u)^{\Box\hspace{-.55em}\lor t}$ for all $u\in [0,1]$ and $t\ge 1$. Furthermore, it is easy to see that the map $\Lambda^\lor:[0,1]\rightarrow[0,1]$ is surjective.

Next, we consider $\Lambda^\lor$ as a map on $\Delta$. For $F\in \Delta$, we define
\begin{align*}
\Lambda^\lor (F)(x):=\Lambda^\lor (F(x)) \qquad x\in \mathbb{R}.
\end{align*}
Note the slight abuse of notation resulting from the dual use of the symbol $\Lambda^\lor$. By definition, we have $\Lambda^\lor (F)\in \Delta$ and the map $\Lambda^\lor :\Delta\rightarrow \Delta$, $F\mapsto \Lambda^\lor (F)$, is an weakly continuous. Moreover, we get the following relations:
\begin{align*}
\Lambda^\lor(FG)&=\Lambda^\lor(F) \Box\hspace{-.95em}\lor \Lambda^\lor(G),\\
\Lambda^\lor(F^t)&=\Lambda^\lor(F)^{\Box\hspace{-.55em}\lor t},
\end{align*}
for $F,G\in \Delta$ and $t\ge1$. 

\begin{remark}
As noted above, the map $\Lambda^\lor: [0,1]\rightarrow[0,1]$ is surjective. However, the map $\Lambda^\lor:\Delta\rightarrow \Delta$ is not surjective. Indeed, for $G\in \Delta$ with $\alpha(G)=-\infty$, we cannot find $F\in \Delta$ such that $G=\Lambda^\lor (F)$ since $\alpha(\Lambda^\lor (F))>-\infty$.
\end{remark}

\begin{proposition}\label{prop:F_t}
For any distribution function $F$ on $\mathbb{R}$, we have $F^{\Box \hspace{-.55em}\lor t} {\Box \hspace{-.70em}\lor} F^{\Box \hspace{-.55em}\lor s}=F^{\Box \hspace{-.55em}\lor (t+s)}$ and $(F^{\Box \hspace{-.55em}\lor t})^{\Box \hspace{-.55em}\lor s}=F^{\Box \hspace{-.55em}\lor ts}$ for any $t,s\ge 1$.
\end{proposition}
\begin{proof}
Fix $x\in\mathbb{R}$. Since $\Lambda^\lor$ is surjective on $[0,1]$, there is a value $u\in [0,1]$ with $F(x)=\Lambda^\lor (u)$. Then $F^{\Box \hspace{-.55em}\lor t}(x)=F(x)^{\Box \hspace{-.55em}\lor t}=\Lambda^\lor (u)^{\Box \hspace{-.55em}\lor t}=\Lambda^\lor(u^t)$ for all $t\ge1$.

We get
\begin{align*}
(F^{\Box \hspace{-.55em}\lor t} {\Box \hspace{-.70em}\lor} F^{\Box \hspace{-.55em}\lor s})(x)&=F^{\Box \hspace{-.55em}\lor t} (x) {\Box \hspace{-.70em}\lor} F^{\Box \hspace{-.55em}\lor s}(x)\\
&=\Lambda^\lor(u^t) {\Box \hspace{-.70em}\lor} \Lambda^\lor(u^s)\\
&=\Lambda^\lor(u^{t+s})=F^{\Box \hspace{-.55em}\lor t+s}(x). 
\end{align*}
Therefore $F^{\Box \hspace{-.55em}\lor t} {\Box \hspace{-.70em}\lor} F^{\Box \hspace{-.55em}\lor s}=F^{\Box \hspace{-.55em}\lor (t+s)}$ for any $t,s\ge 1$.

Next, we obtain
\begin{align*}
(F^{\Box \hspace{-.55em}\lor t})^{\Box \hspace{-.55em}\lor s}(x)&=F^{\Box \hspace{-.55em}\lor t}(x)^{\Box \hspace{-.55em}\lor s}\\
&=\Lambda^\lor(u^t)^{\Box \hspace{-.55em}\lor s}\\
&=\Lambda^\lor(u^{ts})=F^{\Box \hspace{-.55em}\lor ts}(x).
\end{align*}
Hence $(F^{\Box \hspace{-.55em}\lor t})^{\Box \hspace{-.55em}\lor s}=F^{\Box \hspace{-.55em}\lor ts}$ for any $t,s\ge 1$.
\end{proof}

\subsection{Freely max-infinitely divisible distributions}
\label{sec:4}

In this section, we introduce a concept of freely max-infinitely divisible distributions.

\begin{definition}
$F\in \Delta$ is said to be {\it freely max-infinitely divisible} if for each $n\in\mathbb{N}$, there exists $F_n\in \Delta$ such that $F=F_n^{\Box \hspace{-.55em}\lor n}$.
\end{definition}

We give an equivalent property of distribution functions to be freely max-infinitely divisible.

\begin{proposition}\label{prop:FMIDeq}
$F\in \Delta$ is freely max-infinitely divisible if and only if $\alpha(F)>-\infty$.
\end{proposition}

\begin{proof}
If $F$ is freely max-infinitely divisible, then there is a distribution function $F_n$ such that $F=F_n^{\Box \hspace{-.55em}\lor n}$ for each $n\in\mathbb{N}$. Since $\alpha(F_n^{\Box \hspace{-.55em}\lor n})>-\infty$ for every $n\ge 2$, we have that $\alpha(F)>-\infty$. Conversely, if $\alpha(F)>-\infty$, then for each $n\in\mathbb{N}$, we define 
\begin{align*}
F_n(x):=\begin{cases}
\frac{1}{n}F(x)-\left(\frac{1}{n}-1\right), & x\ge \alpha(F)\\
0, & x<\alpha(F).
\end{cases}
\end{align*}
For all $x\ge \alpha(F)$, we have
\begin{align*}
F_n^{\Box \hspace{-.55em}\lor n}(x)=\left(n\left(\frac{1}{n}F(x)-\left(\frac{1}{n}-1\right)\right)-(n-1)\right)_+=F(x).
\end{align*}
Furthermore, $F_n^{\Box \hspace{-.55em}\lor n}(x)=0$ for all $x<\alpha(F)$. Hence $F$ is freely max-infinitely divisible.
\end{proof}

Note that compactly supported distributions are freely max-infinitely divisible while the Gaussian distribution isn't. Moreover, $\Lambda^{\lor}(F)$ is freely max-infinitely divisible for all distribution functions $F$ on $\mathbb{R}$. 



To end this section, we give a limit theorem for every distribution function on $\mathbb{R}$ with respect to free max-convolution.

\begin{proposition}\label{thm:FMID characterization}
Let $F$ be a distribution function on $\mathbb{R}$. Then there exists a sequence $\{F_n\}_n$ of distribution functions on $\mathbb{R}$ such that $\overbrace{F_n \Box \hspace{-.95em}\lor\cdots \Box \hspace{-.95em}\lor F_n}^{n \text{ times}}\xrightarrow{w} F$ as $n\rightarrow\infty$.
\end{proposition}
\begin{proof}

If $F$ is freely max-infinitely divisible, then there is nothing to prove. If $F$ is not freely max-infinitely divisible, then we have $\alpha(F)=-\infty$ by Proposition \ref{prop:FMIDeq}. For each $n\in\mathbb{N}$, we define a distribution function $F_n$ by setting
\begin{align*}
F_n(x):=\begin{cases}
\frac{1}{n} F(x)- \left(\frac{1}{n}-1 \right), & x\ge -n\\
0, & x<-n.
\end{cases}
\end{align*}
Then 
\begin{align*}
\overbrace{F_n \Box \hspace{-.95em}\lor\cdots \Box \hspace{-.95em}\lor F_n}^{n \text{ times}}&=\begin{cases}
F(x), &x\ge -n\\
0, &x<-n
\end{cases}\\
&\xrightarrow{n\rightarrow\infty} F(x),
\end{align*}
for all continuous points $x$ of $F$.
\end{proof}

\begin{remark}
In \cite{BP00}, Bercovici and Pata obtained a free analogue of Khintchine's theorem: Consider a sequence $\{k_n\}_n$ of positive integers with $k_1<k_2<\cdots$, an infinitesimal array $\{\mu_{nk}\}_{1\le k \le k_n, n\ge 1}$ in $\mathcal{P}$ and a sequence $\{a_n\}_n$ of real numbers. If $\mu_{n1}\boxplus \cdots \boxplus \mu_{nk_n}\boxplus \delta_{a_n}$ weakly converges to some probability measure $\mu$ on $\mathbb{R}$, then $\mu\in \mathbf{ID}(\boxplus)$. However, the Khintchine type theorem does not hold in the max-case by Proposition \ref{thm:FMID characterization}.
\end{remark}

\subsection{Free max-stable laws}
\label{sec:5}

In this section, we study spectral distribution functions of
\begin{align*}
\frac{(X_1\lor \cdots \lor X_n)-b_n}{a_n} ,\qquad n\in\mathbb{N},
\end{align*}
for some sequences $\{X_n\}_n$ of freely independent identically distributed real random variables, $\{a_n\}_n\subset (0,\infty)$ and $\{b_n\}_n\subset \mathbb{R}$, where the symbol $\lor$ is the spectral maximum. If $G$ is a spectral distribution function of $X_1$, then one of real random variable $(X_1\lor \cdots \lor X_n-b_n)/a_n$ is equal to $G^{\Box \hspace{-.55em}\lor n}(a_n\cdot +b_n)$ for each $n\in\mathbb{N}$.

\begin{definition}
A non-degenerate distribution function $G$ on $\mathbb{R}$ is said to be {\it freely max-stable} if for each $n\in\mathbb{N}$ there exist $a_n>0$ and $b_n\in\mathbb{R}$ such that
\begin{align*}
G^{\Box \hspace{-.55em}\lor n}(a_n\cdot+b_n)=G(\cdot).
\end{align*}
\end{definition}

Next we define freely max-domain of attraction.

\begin{definition}
A distribution function $F$ is said to be in the {\it free max-domain of attraction} of a distribution function $G$ if there exist $a_n>0$ and $b_n\in\mathbb{R}$ such that
\begin{align*}
F^{\Box \hspace{-.55em}\lor n}(a_n\cdot+b_n)\xrightarrow{w} G(\cdot),
\end{align*}
as $n\rightarrow \infty$. In this case, we write $F\in {\text{Dom}_{\Box \hspace{-.55em}\lor}}(G)$.
\end{definition}

In \cite[Theorem 6.5]{AV06}, Ben Arous and Voiculescu obtain the following equivalent properties of freely max-stable laws.

\begin{proposition}
The following conditions are equivalent.\\
(1) $G$ is freely max-stable; \\
(2) $G \in {\text{Dom}_{\Box \hspace{-.55em}\lor}}(G)$; \\
(3) ${\text{Dom}_{\Box \hspace{-.55em}\lor}}(G)\neq \emptyset$.
\end{proposition}

This means that the freely max-stable laws are weak limits of spectral distribution functions of $(X_1\lor \cdots \lor X_n-b_n)/a_n$, as $n\rightarrow\infty$ for some sequence $\{X_n\}_n$ of freely independent identically distributed real random variables, $\{a_n\}_n\subset (0,\infty)$ and $\{b_n\}_n\subset \mathbb{R}$. In \cite[Theorem 6.8]{AV06}, the freely max-stable laws are characterized by free extreme value distributions.

\begin{theorem}
A non-degenerate distribution function $G$ is freely max-stable if and only if $G$ is of {\it free extreme value type}, that is, there exist $a>0$ and $b\in\mathbb{R}$ such that $G(ax+b)$ is one of the following distributions (called {\it free extreme value distributions}):
\begin{itemize}
\item ${\text F^{\text{free}}_{\text I}}(x):=(1-e^{-x})_+$, \hspace{1mm} $x\in\mathbb{R}$: {\it exponential distribution};
\item ${\text F^{\text{free}}_{\text{II}}}(x):=\begin{cases}
(1-x^{-\alpha})_+, & x>0\\
0, & x\le 0,
\end{cases}$ \hspace{1mm} $\alpha>0$: {\it Pareto distribution};
\item ${\text F^{\text{free}}_{\text{III}}}(x):=
\begin{cases}
1, &x>0\\
(1-|x|^\alpha)_+, &x\le 0,
\end{cases}$ \hspace{1mm} $\alpha>0$: {\it Beta law}.
\end{itemize}
\end{theorem}

\cite[Theorem 6.11-Theorem 6.13]{AV06} states that the classical max-domains of attraction of the extreme value (Gumbel, Fr\'{e}chet and Weibull) distributions are corresponding to the free max-domains of attraction of the free extreme value (exponential, Pareto and Beta) distributions. By using the map $\Lambda^{\lor}$, Ben Arous and Voiculescu get a relation between classical extreme value distributions and freely extreme value distributions, that is, $\Lambda^\lor({\text F}_{n}^{\text{classical}})={\text F}_n^{\text{free}}$ for $n=$ I, II, III in \cite[p.2052]{AV06}. 

In \cite[Theorem 3.3]{BD}, Benaych-Georges and Cabanal-Duvillard obtained a random matrix model whose empirical spectral law weakly converges almost surely to a probability measure constructed by $\Lambda^{\lor}$ as its matrix size goes to infinite.

\subsection{Free regular max-infinitely divisible distributions}
\label{sec:6}

Recall that $\Delta_+$ is the set of all distribution functions on $[0,\infty)$. It is clear that if $F,G\in \Delta_+$ then $F\Box \hspace{-.98em}\lor G$ and $\Lambda^\lor (F)$ are in $\Delta_+$. Furthermore, every distribution function $F$ in $\Delta_+$ is freely max-infinitely divisible since $\alpha(F)=0>-\infty$.

 

In particular, we consider the following subset of $\Delta_+$:
\begin{align*}
\Delta_+^{(0)}:=\{F\in \Delta_+: F(0)=0\}.
\end{align*}
We can see that $\Delta_+^{(0)}$ is the set of all distribution functions of strictly positive random variables.  By the definition, the set $\Delta_+^{(0)}$ is preserved under free max-convolution, that is, if $F,G\in \Delta_+^{(0)}$, then $F \Box \hspace{-.95em} \lor G \in \Delta_+^{(0)}$. Moreover, we have $\alpha(F)\ge 0$ for $F\in \Delta_+^{(0)}$.

Finally, we define free max-infinite divisibility of distribution function in $\Delta_+^{(0)}$. A distribution function $F\in \Delta_+^{(0)}$ is said to be {\it free regular max-infinitely divisible} if for each $n\in\mathbb{N}$, there is $F_n\in \Delta_+^{(0)}$ such that $F=F_n^{\Box\hspace{-.55em}\lor n}$. Note that for any $F\in \Delta_+^{(0)}$, we have $\alpha(F^{\Box \hspace{-.55em} \lor n})>0$. By using similar argument as in the proof of Proposition \ref{prop:FMIDeq}, we can prove that $F\in \Delta_+^{(0)}$ is free regular max-infinitely divisible if and only if $\alpha(F)>0$. For example, the Pareto distribution is free regular max-infinitely divisible. It is easy to see that if $F\in\Delta_+^{(0)}$ then $\Lambda^\lor (F)\in \Delta_+^{(0)}$ is free regular max-infinitely divisible.

\subsection{Free max-analogue of the compound Poisson distribution}

Recall the max-analogue of the compound Poisson distribution (which has already been appeared in \cite[Example 1]{BR77}). Let $\nu\in\mathcal{P}$ with $\nu(\{0\})=0$ and $\lambda\ge 0$. Suppose a distribution function $G$ of $\nu$. The max-analogue of the compound Poisson distribution $\Pi_{\lambda,G}^{\lor}$ is defined as the weak limit of $F_N^{\lor N}$ as $N\rightarrow\infty$, where
\begin{align*}
F_N(x)=\begin{cases}
\left(1-\frac{\lambda}{N}\right)+\frac{\lambda}{N}G(x), & x\ge 0\\
\frac{\lambda}{N}G(x), & x<0,
\end{cases}
\qquad \lambda\le N.
\end{align*} 
The distribution is explicitly written by
\begin{align*}
\Pi_{\lambda,G}^{\lor}(x)=\begin{cases}
\exp\left(-\lambda(1-G(x))\right), &x\ge0\\
0, & x<0.
\end{cases}
\end{align*}

For $\lambda\ge0$ and $G\in\Delta$, we define $\Pi_{\lambda,G}^{\Box \hspace{-.55em}\lor}:=\Lambda^\lor (\Pi_{\lambda,G}^{\lor})\in \Delta_+$. Then it is written by
\begin{align*}
\Pi_{\lambda,G}^{\Box \hspace{-.55em}\lor}(x)=\begin{cases}
(1-\lambda(1-G(x)))_+, &x\ge 0\\
0, & x<0.
\end{cases}
\end{align*}

It is easy to check that $\Pi_{\lambda,G}^{\Box \hspace{-.55em}\lor}$ is the weak limit of $F_N^{\Box \hspace{-.55em} \lor N}$ as $N\rightarrow\infty$. Therefore we say that $\Pi_{\lambda,G}^{\Box \hspace{-.55em}\lor}$ is the free max-analogue of the compound Poisson distribution.
In \cite{HW18}, Huang and Wang have already obtained the bi-free max-case of the compound Poisson distribution.

\section{Theory of Boolean max-infinitely divisible distributions}
\label{sec:7}

All materials in the following discussions are based on \cite{VV18}. We can construct two families $(\tilde{X_i})_{i\in I}$, $(\tilde{Y_j})_{j\in J}$ of Boolean independent bounded noncommutative real random variables (selfadjoint operators) on some Hilbert space $(\mathcal{H},\xi)$ equipped with a unit vector $\xi\in\mathcal{H}$ as follow. 

Consider two families $(X_i)_{i\in I}$ and $(Y_j)_{j\in J}$ of bounded selfadjoint operators on $(\mathcal{H}_1,\xi_1)$ and  $(\mathcal{H}_2,\xi_2)$, respectively. If we consider a Hilbert space $(\mathcal{H},\xi)$ equipped with a unit vector $\xi\in\mathcal{H}$, where $\mathcal{H}:=(\mathcal{H}_1\ominus \mathbb{C}\xi_1)\oplus (\mathcal{H}_2\ominus \mathbb{C}\xi_2) \oplus \mathbb{C}\xi$ and identifying isometries $V_1:\mathcal{H}_1\rightarrow \mathcal{H}$ and $V_2:\mathcal{H}_2\rightarrow \mathcal{H}$, then we have
\begin{align*}
V_1|_{\mathcal{H}_1\ominus\mathbb{C}\xi_1}&=I_{\mathcal{H}_1\ominus\mathbb{C}\xi_1}, \qquad V_1\xi_1=\xi;\\
V_2|_{\mathcal{H}_2\ominus\mathbb{C}\xi_2}&=I_{\mathcal{H}_2\ominus\mathbb{C}\xi_2}, \qquad V_2\xi_2=\xi.
\end{align*}
Consider $\tilde{X_i}:=V_1X_iV_1^*$ and $\tilde{Y_j}:=V_2Y_jV_2^*$ for all $i\in I$ and $j\in J$. Then $(\tilde{X_i})_{i\in I}$ and $(\tilde{Y_j})_{j\in J}$ are Boolean independent bounded selfadjoint operators on $(\mathcal{H},\xi)$. Note that
\begin{align*}
\langle (\tilde{X_i})^n\xi,\xi \rangle_\mathcal{H}=\langle (X_i)^n\xi_1,\xi_1 \rangle_{\mathcal{H}_1}, \qquad \langle (\tilde{Y_j})^n\xi,\xi \rangle_\mathcal{H}=\langle (Y_j)^n\xi_2,\xi_2 \rangle_{\mathcal{H}_2},
\end{align*}
for any $i\in I$, $j\in J$ and $n\in\mathbb{N}$.

This discussion works for unbounded real random variables (selfadjoint operators) by replacing operators to spectral scales. However, $\mathbf{E}(\tilde{X}; (-\infty,t])$ and $\mathbf{E}(X; (-\infty,t])$ are not equivalent even if $\tilde{X}$ and $X$ are equivalent, where $X$ and $Y$ are said to be {\it equivalent} if there exists an isometry $V$ such that $Y=VXV^*$. Actually, we have
\begin{align*}
\mathbf{E}(\tilde{X_i}; (-\infty,t])&=V_1\mathbf{E}(X_i; (-\infty,t])V_1^*, \qquad t<0,\\
\mathbf{E}(\tilde{X_i}; (-\infty,t])&=V_1\mathbf{E}(X_i; (-\infty,t])V_1^*+P_{\mathcal{H}_2\ominus \mathbb{C}\xi_2}, \qquad t\ge 0.
\end{align*}
Similarly, we have
\begin{align*}
\mathbf{E}(\tilde{Y_j}; (-\infty,t])&=V_2\mathbf{E}(Y_j; (-\infty,t])V_2^*, \qquad t<0,\\
\mathbf{E}(\tilde{Y_j}; (-\infty,t])&=V_2\mathbf{E}(Y_j; (-\infty,t])V_2^*+P_{\mathcal{H}_1\ominus \mathbb{C}\xi_1}, \qquad t\ge 0.
\end{align*}
Since we have
\begin{align*}
P:=(V_1\mathbf{E}(X_i, (-\infty,t])V_1^*)\land (V_2\mathbf{E}(Y_j, (-\infty,t])V_2^*)\le P_{\mathbb{C}\xi},
\end{align*}
the projection $P$ is either $0$ or $P_{\mathbb{C}\xi}$, and therefore $\langle P\xi,\xi\rangle_\mathcal{H}\in \{0,1\}$. On the other hand, the projection
\begin{align*}
Q:=(V_1\mathbf{E}(X_i, (-\infty,t])V_1^*+P_{\mathcal{H}_2\ominus \mathbb{C}\xi_2} )\land (V_2\mathbf{E}(Y_j, (-\infty,t])V_2^*+P_{\mathcal{H}_1\ominus \mathbb{C}\xi_1})
\end{align*}
is neither $0$ nor $P_{\mathbb{C}\xi}$, and therefore $\langle Q\xi,\xi\rangle_{\mathcal{H}}\in [0,1]$ takes non-trivial value.

Finally, the non-trivial spectral max works when we choose $X\ge 0$ and $Y\ge0$. Strictly speaking, the spectral projections $\mathbf{E}(\tilde{X}\lor \tilde{Y}; (-\infty,t])$ and $\mathbf{E}(\tilde{X}\land \tilde{Y};(-\infty,t])$ are neither $0$ nor $P_{\mathbb{C}\xi}$ if $X$ and $Y$ are positive. Therefore, we assume positivity of real random variables when we consider the maximum of Boolean independent real random variables.

\subsection{Boolean max-convolution}
\label{sec:8}

Consider a Hilbert space $(\mathcal{H},\xi)$ equipped with a unit vector $\xi\in\mathcal{H}$. Firstly we define Boolean max-convolution of Boolean independent projections on $(\mathcal{H},\xi)$. Define the vector state $\varphi(X):=\langle X\xi,\xi \rangle_\mathcal{H}$ on the set of all operators on $(\mathcal{H},\xi)$. If $P$ is a projection on $(\mathcal{H},\xi)$, then its spectral distribution function $F_P$ is 
\begin{align*}
F_P(x)=\begin{cases}
0, & x<0\\
p, & 0\le x <1\\
1, & x\ge 1,
\end{cases}
\end{align*}
where $p=\varphi(I-P)\in [0,1]$. In \cite[Lemma 3.1]{VV18}, Vargas and Voiculescu have the following statement.

\begin{lemma}\label{lem:Bip}
If $P,Q$ are Boolean independent projections on $(\mathcal{H},\xi)$ with $\varphi(P)=1-p$ and $\varphi(Q)=1-q$, then $\varphi(P\lor Q)=1-r$, where $r\in [0,1]$ satisfies that
\begin{align*}
r^{-1}-1=(p^{-1}-1)+(q^{-1}-1).
\end{align*}
\end{lemma}
If either $p$ or $q$ is equal to $0$, then we define $r=0$. Due to Lemma \ref{lem:Bip}, we can define a semigroup structure on $[0,1]$. We define an operation $\cup \hspace{-.67em}\lor$ on $[0,1]$ by setting
\begin{align*}
(x \cup \hspace{-.88em}\lor y)^{-1}-1:=(x^{-1}-1)+(y^{-1}-1), \qquad x,y\in [0,1],
\end{align*}
where we understand $x \cup \hspace{-.88em}\lor  y=0$ if either $x$ or $y$ is equal to $0$.


Considering the above discussion, we define Boolean max-convolution as follow.

\begin{definition}
Consider $F,G\in \Delta_+$. We define a distribution function $F\cup \hspace{-.80em}\lor G$ as follows:
\begin{align*}
(F\cup \hspace{-.88em}\lor G)(x):=F(x) \cup \hspace{-.88em}\lor  G(x).
\end{align*}
The operation $\cup \hspace{-.68em}\lor:\Delta_+\times \Delta_+\rightarrow \Delta_+$, $(F,G)\mapsto F\cup \hspace{-.88em}\lor G$, is called {\it Boolean max-convolution}. Note the slight abuse of notation resulting from the dual use of the symbol $\cup \hspace{-.68em}\lor$.
\end{definition}

\begin{remark}
Recall that $\Delta_+$ is the set of all distribution functions on $[0,\infty)$. We can regard $\Delta_+$ as the set of functions $G:[0,\infty)\rightarrow [0,1]$ which are increasing, right-continuous and satisfy $\lim_{x\rightarrow\infty}G(x)=1$ by restricting a distribution function $F\in \Delta_+$ to a function $F|_{[0,\infty)}$. In \cite{VV18}, Vargas and Voiculescu use this identification to discuss extreme value theory in Boolean setting. However, the reader notes not to use this identification in our paper.
\end{remark}

In \cite[Lemma 3.3]{VV18}, if $X,Y\ge0$ are Boolean independent positive random variables on $(\mathcal{H},\xi)$, then $F_{X\lor Y}=F_X \cup \hspace{-.73em}\lor F_Y$, where $F_X,F_Y$ and $F_{X\lor Y}$ mean spectral distribution functions. Thus, Boolean max-convolution realizes the spectral distribution function of maximum of Boolean independent positive random variables.

By definition of Boolean max-convolution, for any $F, G \in \Delta_+$, we have
\begin{align*}
F\cup \hspace{-.88em}\lor G=\frac{FG}{F+G-FG}\in\Delta_+,
\end{align*}
and
\begin{align*}
F^{\cup \hspace{-.50em}\lor n}:= \overbrace{F \cup \hspace{-.88em}\lor \cdots \cup \hspace{-.88em}\lor F}^{n \text{ times}}=\frac{F}{n-(n-1)F}\in\Delta_+, \qquad n\in\mathbb{N}.
\end{align*}

In \cite[Lemma 3.2]{VV18}, two semigroups $([0,1], \cup \hspace{-.66em}\lor )$ and $([0,1], \cdot)$ are isomorphic by an order preserving isomorphism $\mathcal{X}:[0,1]\rightarrow [0,1]$ defined by
\begin{align*}
\mathcal{X}(u):&=\exp(1-u^{-1}), \qquad u\in (0,1],\\
\mathcal{X}(0):&=0.
\end{align*}
Note that $\mathcal{X}(u^{\cup \hspace{-.50em}\lor n})=\mathcal{X}(u)^n$ for all $u\in [0,1]$ and $n\in\mathbb{N}$. It is easy to see that $\mathcal{X}^{-1}$ is given by
\begin{align*}
\mathcal{X}^{-1}(u)&=\frac{1}{1-\log u}, \qquad u\in (0,1],\\
\mathcal{X}^{-1}(0)&=0.
\end{align*} 
Obviously, we have $\mathcal{X}^{-1}(u^n)=\mathcal{X}^{-1}(u)^{\cup \hspace{-.50em}\lor n}$ for all $u\in[0,1]$ and $n\in\mathbb{N}$.

Next we define $\mathcal{X}$ and $\mathcal{X}^{-1}$ as maps on $\Delta_+$. For a distribution function $F\in \Delta_+$, we define
\begin{align*}
\mathcal{X}(F)(x):=\mathcal{X}(F(x)) \text{ and }  \mathcal{X}^{-1}(F)(x):=\mathcal{X}^{-1}(F(x)), \qquad x\in \mathbb{R}.
\end{align*} 
Note the slight abuse of notation resulting from the dual use of the symbols $\mathcal{X}$ and $\mathcal{X}^{-1}$. We then have 
\begin{align*}
\mathcal{X}(F\cup \hspace{-.80em}\lor G)&=\mathcal{X}(F)\mathcal{X}(G),\qquad \mathcal{X}(F^{\cup \hspace{-.50em}\lor n})=\mathcal{X}(F)^n,
\end{align*}
and
\begin{align*}
\mathcal{X}^{-1}(FG)&=\mathcal{X}^{-1}(F) \cup \hspace{-.80em}\lor\mathcal{X}^{-1}(G), \qquad \mathcal{X}^{-1}(F^n)=\mathcal{X}^{-1}(F)^{\cup \hspace{-.50em}\lor n},
\end{align*}
for distribution functions $F,G\in \Delta_+$ and $n\in\mathbb{N}$.

\subsection{Boolean max-infinitely divisible distributions}
\label{sec:9}

In this section, we introduce Boolean max-infinitely divisible distributions and semigroup of distribution functions with respect to Boolean max-convolution.

\begin{definition}
$F\in \Delta_+$ is said to be {\it Boolean max-infinitely divisible} if for each $n\in\mathbb{N}$ there exists $F_n\in \Delta_+$ such that $F=F_n^{\cup \hspace{-.50em}\lor n}$.
\end{definition}

We study the class of Boolean max-infinitely divisible distribution functions.

\begin{proposition}\label{thm:BMID}
Every distribution function in $\Delta_+$ is Boolean max-infinitely divisible.
\end{proposition}

\begin{proof}
Consider a distribution function $F\in \Delta_+$. Fix $x\in\mathbb{R}$. Then there is a value $u\in [0,1]$ with $F(x)=\mathcal{X}^{-1}(u)$ since $\mathcal{X}^{-1}$ (and therefore $\mathcal{X}$) is bijection on $[0,1]$. For each $n\in\mathbb{N}$, we get
\begin{align*}
F(x)=\mathcal{X}^{-1}(u)=\mathcal{X}^{-1}(u^{1/n})^{\cup \hspace{-.50em}\lor n}=\mathcal{X}^{-1}(F^{1/n}(x))^{\cup \hspace{-.50em}\lor n}.
\end{align*}
Therefore we obtain $F=\mathcal{X}^{-1}(F^{1/n})^{\cup \hspace{-.50em}\lor n}$.
\end{proof}

This proposition is similar to the fact that every probability measure is (additive) Boolean infinitely divisible (see \cite[Theorem 3.6]{SW}). For any $F\in\Delta_+$ and each $n\in\mathbb{N}$, a choice of the $n$-th root of $F$ is unique. 

We can explicitly write the distribution $\mathcal{X}^{-1}(F^{1/n})$ as follows:
\begin{align*}
\mathcal{X}^{-1}(F^{1/n})=\frac{F}{1/n-(1/n-1)F}, \qquad n\in\mathbb{N}.
\end{align*}
For any distribution functions $F\in \Delta_+$, we denote by $F^{\cup \hspace{-.55em}\lor 1/n}:=\mathcal{X}^{-1}(F^{1/n})$ for $n\in\mathbb{N}$. Furthermore, we can extend the above distribution function to $F^{\cup \hspace{-.50em}\lor t}$ for $t\ge0$ as follows:
\begin{align*}
F^{\cup \hspace{-.50em}\lor t}:&=\mathcal{X}^{-1}(F^{t})=\frac{F}{t-(t-1)F}, \qquad t>0,\\
F^{\cup \hspace{-.50em}\lor 0}:&=\mathbf{1}_{[\alpha(F),\infty)}.
\end{align*}
By definition of $\mathcal{X}^{-1}$, the following proposition holds.

\begin{proposition}\label{prop:B_t}
\begin{enumerate}[(1)]
\item The map $[0,\infty)\rightarrow \Delta_+$, $t\mapsto F^{\cup \hspace{-.50em}\lor t}$, is weakly continuous.
\item  We have $F^{\cup \hspace{-.50em}\lor t}\cup \hspace{-.80em} \lor F^{\cup \hspace{-.50em}\lor s}=F^{\cup \hspace{-.50em}\lor t+s}$ and $(F^{\cup \hspace{-.50em}\lor t})^{\cup \hspace{-.50em}\lor s}=F^{\cup \hspace{-.50em}\lor ts}$ for any distribution functions $F\in \Delta_+$ and $t,s\ge0$. 
\end{enumerate}
\end{proposition}

\subsection{Boolean max-stable laws}
\label{sec:10}

In this section, we introduce the Boolean max-stable laws. They are the most important to consider Boolean max-probability theory from a reason that it is analogue of classcal and free max-stable (extreme value) laws which are very important to study extreme value statistics.

\begin{definition}\label{def:BMS}
A non-degenerate distribution function $G\in \Delta_+$ is said to be {\it Boolean max-stable} if for some $F\in\Delta_+$ there exists $a_n>0$ such that $F^{\cup \hspace{-.50em}\lor n}(a_n\cdot)\xrightarrow{w} G(\cdot)$. 
\end{definition}

The Boolean max-stable laws are analogue of classical/freely max-stable laws. This is the weak limit of spectral distribution functions of 
\begin{align*}
\frac{X_1\lor \cdots \lor X_n}{a_n}, \qquad n\in\mathbb{N},
\end{align*}
for some sequences $\{X_n\}_n$ of Boolean independent identically distributed positive random variables, $\{a_n\}_n\subset (0,\infty)$, where the symbol $\lor$ means the spectral maximum.

Looking at the map $\mathcal{X}$ on $\Delta_+$ and the classical max-stable laws, we obtain the following important theorem.

\begin{theorem}\label{lem:b-max-stable} (See \cite[Theorem 4.1, Corollary 4.1]{VV18})
A non-degenerate distribution function $G\in \Delta_+$ is Boolean max-stable if and only if
\begin{align*}
\mathcal{X}(G)(x)=\begin{cases}
\exp(-\lambda x^{-\alpha}), & x>0\\
0, & x\le 0
\end{cases}
\end{align*}
for some $\lambda,\alpha>0$. Equivalently, 
\begin{align*}
G(x)=\begin{cases}
(1+\lambda x^{-\alpha})^{-1}, & x>0\\
0, & x\le 0
\end{cases}
\end{align*}
for some $\lambda,\alpha>0$.
\end{theorem}

By Theorem \ref{lem:b-max-stable}, every Boolean max-stable law is characterized by two parameters $\lambda>0$ and $\alpha>0$, so that we write it as
\begin{align*}
{\bf D}_{\lambda,\alpha}(x):=\begin{cases}
(1+\lambda x^{-\alpha})^{-1}, & x>0\\
0, & x\le 0.
\end{cases}
\end{align*}
In particular, we write ${\bf D}_\alpha:={\bf D}_{1,\alpha}$. The distribution ${\bf D}_{\lambda,\alpha}$ is also called the {\it Dagum distribution}.

We define Boolean max-domain of attraction of distribution functions.

\begin{definition}
$F\in \Delta_+$ is said to be in the {\it Boolean max-domain of attraction} of $G\in \Delta_+$ if for some sequence of $a_n>0$, we have $F^{\cup \hspace{-.50em}\lor n}(a_n \cdot)\xrightarrow{w} G(\cdot)$ as $n\rightarrow \infty$. In this case we write $F\in {\text{Dom}_{\cup \hspace{-.50em}\lor}}(G)$.
\end{definition}

\begin{remark}\label{rem:stable}
For each Boolean max-stable law $G={\bf D}_{\lambda,\alpha}$, we can take $F=G$ and $a_n=n^{1/\alpha}$ in Definition \ref{def:BMS}. From the reason, if $G$ is Boolean max-stable law, then $G\in {\text{Dom}_{\cup \hspace{-.50em}\lor}}(G)$.
\end{remark}


To end this section, we characterize distribution functions which are in the Boolean max-domain of attraction of Dagum distributions by behavior at the tail of distribution function. 

For a distribution function $G$ on $\mathbb{R}$, we define a function $\overline{G}:=1-G$. The function is important to study distribution functions of the maximum of random variables since we have to focus on behavior at the tail to look at statistics of the maximum of random variables. Recall the definition of regularly varying functions.

\begin{definition}
A function $f$ is said to be {\it regularly varying of index $\alpha$} if for all $t>0$ we have that $f(tx)/f(x)\rightarrow t^\alpha$ as $x\rightarrow \infty$. In particular, if $\alpha=0$, then $f$ is said to be {\it slowly varying}.
\end{definition}

From \cite[Theorem 6.12]{AV06} and \cite[Corollary 4.3]{VV18}, we obtain the following theorem.

\begin{theorem}\label{lem:max-boolean domain of attraction}
The following conditions are equivalent for $\alpha>0$.\\
(1) $G \in {\text{Dom}_{\cup \hspace{-.50em}\lor}}({\bf D_\alpha})$; \\
(2) $G$ is in the classical max-domain of attraction of the Fr\'{e}chet distribution $\Phi_\alpha$:
\begin{align*}
\Phi_\alpha(x):={\text F_{\text{II}}^{\text{classical}}}(x)=
\begin{cases}
\exp(-x^{-\alpha}), & x>0\\
0, & x\le0
\end{cases};
\end{align*}
(3) $G$ is in the free max-domain of attraction of the Pareto distribution ${\bf P}_\alpha$ with exponent $\alpha$:
\begin{align*}
{\bf P}_\alpha(x):={\text F_{\text{II}}^{\text{free}}}(x)=
\begin{cases}
(1-x^{-\alpha})_+, & x>0\\
0, & x\le 0
\end{cases} ;
\end{align*}
(4) $\overline{G}$ is regularly varying of index $-\alpha$. 
\end{theorem}

\subsection{Tails of max-convolution power of distribution functions}
\label{sec:11}

Since $\Phi_\alpha(n^{-1/\alpha}\cdot)=\Phi_\alpha^{\lor n}$, $\mathbf{P}_\alpha(n^{-1/\alpha} \cdot)=\mathbf{P}_\alpha^{\Box \hspace{-.55em} \lor n}$ and ${\bf D_\alpha}(n^{-1/\alpha}\cdot)=\mathbf{D}_\alpha^{\cup \hspace{-.50em} \lor n}$ for all $\alpha>0$ and $n\in\mathbb{N}$, the classical, free and Boolean domains of attraction of $\Phi_\alpha$, $\mathbf{D}_\alpha$ and $\mathbf{P}_\alpha$ coincide with ones of $\Phi_\alpha^{\lor n}$, $\mathbf{P}_\alpha^{\Box \hspace{-.55em} \lor n}$ and $\mathbf{D}_\alpha^{\cup \hspace{-.50em} \lor n}$, respectively. Therefore, $\Phi_\alpha^{\lor n}$, $\mathbf{P}_\alpha^{\Box \hspace{-.55em} \lor n}$ and $\mathbf{D}_\alpha^{\cup \hspace{-.50em} \lor n}$ are regularly varying of index $-\alpha$ by Theorem \ref{lem:max-boolean domain of attraction}. This means that three type max-convolutions preserve tails of corresponding extreme value distributions. More generally, we can conclude that three type max-convolution preserve tails of distribution functions as follows. However, the following proposition has already been obtained in \cite{CH18} and \cite{HM13} to study behavior at tails of free and Boolean subexponential distributions, but for readers convenience we include the proof.

\begin{proposition}\label{thm:tail} Consider $n\in\mathbb{N}$. We have the following conditions.
\begin{enumerate}[(1)]
\item Let $F$ be a distribution function. Then $\overline{F^{\lor n}}\sim n\overline{F}$ as $x\rightarrow\infty$. 
\item Let $F$ be a distribution function. Then $\overline{F^{\Box \hspace{-.55em} \lor n} }\sim n\overline{F}$ as $x\rightarrow\infty$.
\item Let $F\in\Delta_+$. Then $\overline{F^{\cup \hspace{-.50em} \lor n}}\sim n\overline{F}$ as $x\rightarrow\infty$.
\end{enumerate}
\end{proposition}

\begin{proof}
(1) As $x\rightarrow\infty$, we have $F\rightarrow1$. Therefore
\begin{align*}
\overline{F^{\lor n}}=1-F^n=n\overline{F}(1+o(1)), \qquad \text{ as } x\rightarrow\infty. 
\end{align*}
Hence we have $\overline{F^{\lor n}}\sim n\overline{F}$ as $x\rightarrow \infty$.

(2) We have
\begin{align*}
\overline{F^{\Box \hspace{-.55em} \lor n} }=1- \max\{ nF-(n-1), 0\}=\min \{n(1-F), 1\}.
\end{align*}
As $x\rightarrow \infty$, we may consider $n(1-F)<1$. Hence we have $\overline{F^{\Box \hspace{-.55em} \lor n} }\sim n\overline{F}$ as $x\rightarrow\infty$.

(3) We have
\begin{align*}
\overline{F^{\cup \hspace{-.50em} \lor n}}= 1-\frac{F}{n-(n-1)F}=\frac{n(1-F)}{n-(n-1)F}.
\end{align*}
Since $n-(n-1)F\rightarrow 1$ as $x\rightarrow\infty$, we have $\overline{F^{\cup \hspace{-.50em} \lor n}}\sim n\overline{F}$ as $x\rightarrow\infty$.
\end{proof}

In the same way to prove Proposition \ref{thm:tail}, it is easy to get a generalization of the above proposition as follows.

\begin{corollary}\label{cor:tF}
We have the following conditions.
\begin{enumerate}[(1)]
\item Let $F$ be a distribution function and $t>0$. Then $\overline{F^{\lor t}}\sim t\overline{F}$ as $x\rightarrow\infty$. 
\item Let $F$ be a distribution function and $t\ge 1$. Then $\overline{F^{\Box \hspace{-.55em} \lor t} }\sim t\overline{F}$ as $x\rightarrow\infty$.
\item Let $F\in\Delta_+$ and $t>0$. Then $\overline{F^{\cup \hspace{-.50em} \lor t}}\sim t\overline{F}$ as $x\rightarrow\infty$.
\end{enumerate}
\end{corollary}

In particular, we get a relation between three type max-convolutions and regularly varying functions by using Corollary \ref{cor:tF}. Note that the following relation has already been obtained in \cite[Proposition 2.4]{CH18}  in the case when $t$ is a positive integer.

\begin{corollary}\label{cor:F,r.v.a}
Consider $\alpha \in\mathbb{R}$. The following conditions are equivalent.
\begin{enumerate}[(1)]
\item $\overline{F}$ is regularly varying of index $\alpha$;
\item $\overline{F^{\lor t}}$ is regularly varying of index $\alpha$ for some (any) $t>0$;
\item $\overline{F^{\Box \hspace{-.55em} \lor t} }$ is regularly varying of index $\alpha$ for some (any) $t\ge 1$;
\item $\overline{F^{\cup \hspace{-.50em} \lor t}}$ is regularly varying of index $\alpha$ for some (any) $t>0$.
\end{enumerate}
\end{corollary}

\section{Max-Belinschi-Nica semigroup}
\label{sec:12}

In \cite[Proposition 3.1]{BN08}, there is a relation between free (additive) convolution and Boolean (additive) convolution, that is, for $p\ge1$ and $q>1-1/p$,
\begin{align*}
(\mu^{\boxplus p})^{\uplus q}=(\mu^{\uplus q'})^{\boxplus p'}, \qquad \mu\in\mathcal{P},
\end{align*}
where 
\begin{align}\label{pq}
\begin{cases}
p'=pq/(1-p+pq)\\
q'=1-p+pq,
\end{cases}  \text{ equivalently, }\hspace{2mm}
\begin{cases}
p=1-q'+p'q'\\
q=p'q'/(1-q'+p'q').
\end{cases}
\end{align}
Note that $p'\ge1$ and $q'>0$. By using free and Boolean max-convolutions, we obtain a completely analogue of the above relation.

\begin{proposition}\label{thm:m-fb-convolution}
For any $F\in\Delta_+$, $p\ge1$ and $q>1-1/p$, we have
\begin{align*}
(F^{\Box \hspace{-.55em}\lor p})^{\cup \hspace{-.50em}\lor q}=(F^{\cup \hspace{-.50em}\lor q'})^{\Box \hspace{-.55em}\lor p'},
\end{align*}
where $p',q'$ satisfy the equation \eqref{pq}. 
\end{proposition}

\begin{proof}
Note that for any $F\in\Delta_+$, $p\ge 1$ and $q>1-1/p$,
\begin{align*}
q+(1-q)(pF-(p-1))&=p(q-1)(1-F)+1\\
&>p\left(1-\frac{1}{p}-1 \right)(1-F)+1=F\ge0.
\end{align*}
Therefore, for any $F\in\Delta_+$, $p\ge1$ and $q>1-1/p$, we have
\begin{align*}
(F^{\Box \hspace{-.55em}\lor p})^{\cup \hspace{-.50em}\lor q}&=\frac{F^{\Box \hspace{-.55em}\lor p}}{q+(1-q)F^{\Box \hspace{-.55em}\lor p}}=\frac{(pF-(p-1))_+}{q+(1-q)(pF-(p-1))_+}\\
&=\begin{cases}
\frac{pF-(p-1)}{q+(1-q)(pF-(p-1))}, & \text{ if } pF-(p-1)\ge 0\\
0, & \text{ if } pF-(p-1)<0
\end{cases}\\
&=\left( \frac{pF-(p-1)}{q+(1-q)(pF-(p-1))}\right)_+.
\end{align*}
On the other hand, we have
\begin{align*}
(F^{\cup \hspace{-.50em}\lor q'})^{\Box \hspace{-.55em}\lor p'}&=\left(\frac{F}{q'+(1-q')F}\right)^{\Box \hspace{-.55em}\lor p'}=\left(\frac{p'F}{q'+(1-q')F}-(p'-1)\right)_+\\
&=\left(\frac{(1-q'+p'q')F-(p'-1)q'}{q'+(1-q')F}\right)_+\\
&=\left(\frac{pF-(p-1)}{q+(1-q)(pF-(p-1))}\right)_+,
\end{align*}
where the last equation holds by using \eqref{pq}. Thus the formula of this theorem holds.
\end{proof}

Recall the Belinschi-Nica semigroup in Section 1. Similarly, we define a Belinschi-Nica type semigroup for free and Boolean max-convolutions.

\begin{proposition}\label{thm:MBNs}
For each $t\ge0$, we define a map ${\bf B}_t^M:\Delta_+\rightarrow \Delta_+$ as
\begin{align*}
{\bf B}_t^M(F):=\left(F^{\Box \hspace{-.55em}\lor (1+t)}\right)^{\cup \hspace{-.50em}\lor \frac{1}{1+t}},\qquad  F\in\Delta_+.
\end{align*}
The family $\{{\bf B}_t^M\}_{t\ge0}$ is a semigroup with respect to composition and $\alpha({\bf B}_t^M(F))=\alpha(F^{\Box \hspace{-.55em}\lor (1+t)})$. In particular, $\mathbf{B}_t^M(F)$ is free regular max-infinitely divisible for any $F\in \Delta_+^{(0)}$ and $t>0$.  
\end{proposition}
\begin{proof}
Using Propositions \ref{prop:F_t}, \ref{prop:B_t} and \ref{thm:m-fb-convolution}, for any $F\in\Delta_+$, we have
\begin{align*}
{\bf B}_s^M\circ {\bf B}_t^M(F)&={\bf B}_s\left( (F^{\Box \hspace{-.55em}\lor (1+t)})^{\cup \hspace{-.50em}\lor \frac{1}{1+t}}\right)=\left[\left((F^{\Box \hspace{-.55em}\lor (1+t)})^{\cup \hspace{-.50em}\lor \frac{1}{1+t}}\right)^{\Box \hspace{-.55em}\lor (1+s)}\right]^{\cup \hspace{-.50em}\lor \frac{1}{1+s}}\\
&=\left[\left( (F^{\Box \hspace{-.55em}\lor (1+t)})^{\Box \hspace{-.50em}\lor \frac{1+t+s}{1+t}}\right)^{\cup \hspace{-.50em}\lor \frac{1+s}{1+t+s}}\right]^{\cup \hspace{-.50em}\lor \frac{1}{1+s}}=\left(F^{\Box \hspace{-.55em}\lor (1+t+s)}\right)^{\cup \hspace{-.50em}\lor \frac{1}{1+t+s}}={\bf B}_{t+s}^M(F).
\end{align*}

By using the fact $\alpha(F^{\cup \hspace{-.50em}\lor t})=\alpha(F)$ for every $t>0$, we have $\alpha({\bf B}_t^M(F))=\alpha(F^{\Box \hspace{-.55em}\lor (1+t)})$.

If $F(0)=0$, then $\mathbf{B}_t^M(F(0))=\mathbf{B}_t^M(0)=0$. Hence $\mathbf{B}_t^M(\Delta_+^{(0)})\subset \Delta_+^{(0)}$ holds. For any $t>0$, we have $\alpha(F^{\Box \hspace{-.55em} \lor(1+t)})>0$ by right continuity of $F$ and $F(0)=0$. Therefore $\mathbf{B}_t^M(F)$ is free regular max-infinitely divisible for any $F\in \Delta_+^{(0)}$.
\end{proof}

\begin{definition}
The semigroup $\{{\bf B}_t^M\}_{t\ge0}$ is called the {\it max-Belinschi-Nica semigroup}. For each $t\ge 0$, the map ${\bf B}_t^M$ is called the {\it max-Belinschi-Nica map at time $t$}.
\end{definition}

In \cite[Theorem 2.5]{CH18}, we know a relation between behavior at tails of $\mu\in\mathcal{P}$ and one of $\mathbf{B}_t(\mu)$. In the max-case, by applying Corollary \ref{cor:F,r.v.a}, we can conclude that the map $\mathbf{B}_t^M$ preserves behavior at tails of distribution functions.

\begin{corollary}
Let $F\in \Delta_+$ and $t\ge 0$. Then $\overline{\mathbf{B}_t^M(F)}\sim \overline{F}$ as $x\rightarrow\infty$.
\end{corollary}

Set $\Theta_+:=\Lambda^\lor (\Delta_+)\subset \Delta_+$. Recall that $\Lambda^\lor (\Phi_\alpha)=\mathbf{P}_\alpha$, where $\Phi_\alpha$ and $\mathbf{P}_\alpha$ are the Fr\'{e}chet distribution and the Pareto distribution, respectively. We show that the Belinschi-Nica map $\mathbf{B}_1^M$ has rich properties as follow.

\begin{proposition}\label{prop:B_1}
The following conditions hold:
\begin{enumerate}[(1)]
\item We have ${\bf B}_1^M=\Lambda^{\lor}\circ \mathcal{X}$, where the map $\mathcal{X}$ was defined Section 4.1.
\item ${\bf B}_1^M$ is a surjection from $\Delta_+$ to $\Theta_+$. 
\item We have ${\bf B}_1^M({\bf D_\alpha})={\bf P_\alpha}$ for any $\alpha>0$.
\item We have ${\bf B}_1^M(F\cup \hspace{-.80em}\lor G)={\bf B}_1^M(F)\Box \hspace{-.85em}\lor  {\bf B}_1^M(G)$ and ${\bf B}_1^M(F^{\cup \hspace{-.50em}\lor t})={\bf B}_1^M(F)^{\Box \hspace{-.55em}\lor t}$ for any $F,G\in \Delta_+$ and $t\ge 1$.
\end{enumerate}
\end{proposition}

\begin{proof}
(1) For all $F\in\Delta_+$, we have that
\begin{align*}
{\bf B}_1^M(F)&=\frac{(2F-1)_+}{1/2+(1/2)(2F-1)_+}=\begin{cases}
2-\frac{1}{F}, & \text{ if } 1/2\le F\le 1\\
0, & \text{ if } 0\le F <1/2
\end{cases}\\
&=\left(2-\frac{1}{F}\right)_+=(1+\log\mathcal{X}(F))_+=\Lambda^{\lor}(\mathcal{X}(F)).
\end{align*}

(2) By (1), the map $\mathbf{B}_1^M$ takes values in $\Theta_+$. For any $G\in \Theta_+$, there is $H\in\Delta_+$ such that $G=(1+\log H)_+$, where the reader understands $\log H(x)=-\infty$ if $H(x)=0$ for $x\in \mathbb{R}$. Put $F:=1/(1-\log H)\in \Delta_+$. Then we have
\begin{align*}
{\bf B}_1^M(F)&=\left( 2F-1\right)_+^{\cup \hspace{-.50em}\lor \frac{1}{2}}=\left(\frac{1+ \log H}{1-\log H}\right)_+^{\cup \hspace{-.50em}\lor \frac{1}{2}}\\
&=\frac{\left(\frac{1+ \log H}{1-\log H}\right)_+}{\frac{1}{2}+\frac{1}{2}\times \left(\frac{1+ \log H}{1-\log H}\right)_+}=(1+\log H)_+=G.
\end{align*}
Hence ${\bf B}_1^M$ is surjective from $\Delta_+$ to $\Theta_+$. \\

(3) Note that $\mathcal{X}(\mathbf{D}_\alpha)=\Phi_\alpha$ (see Section 4.3). By (1), we have
\begin{align*}
\mathbf{B}_1^M(\mathbf{D}_\alpha)=\Lambda^\lor\circ \mathcal{X}(\mathbf{D}_\alpha)=\Lambda^\lor (\Phi_\alpha)=\mathbf{P}_\alpha.
\end{align*}

(4) By definition of $\Lambda^{\lor}$ and $\mathcal{X}$, we have
\begin{align*}
{\bf B}_1^M(F\cup \hspace{-.88em}\lor G)&=\Lambda^{\lor}(\mathcal{X}(F\cup \hspace{-.88em}\lor G))=\Lambda^{\lor}(\mathcal{X}(F)\mathcal{X}(G))\\
&=\Lambda^{\lor}(\mathcal{X}(F))\Box \hspace{-.95em}\lor \Lambda^{\lor}(\mathcal{X}(G))={\bf B}_1^M(F)\Box \hspace{-.95em}\lor  {\bf B}_1^M(G),
\end{align*}
and
\begin{align*}
{\bf B}_1^M(F^{\cup \hspace{-.50em}\lor t})=\Lambda^\lor (\mathcal{X}(F^{\cup \hspace{-.50em}\lor t}))=\Lambda^\lor (\mathcal{X}(F)^t)=\Lambda^\lor(\mathcal{X}(F))^{\Box\hspace{-.55em}\lor t}={\bf B}_1^M(F)^{\Box\hspace{-.55em}\lor t}
\end{align*}
for any $F,G\in\Delta_+$ and $t\ge 1$.
\end{proof}

In addition, we have $\mathbf{B}_1^M(\Delta_+)=\Theta_+$ by Proposition \ref{prop:B_1} (1) and (2). Finally, we conclude that the map ${\bf B}_1^M$ connects limit theorems for Boolean and freely max-infinitely divisible distributions.

\begin{theorem}\label{thm:limit}
Let $\{F_n\}_n$ be a sequence in $\Delta_+$ and $\{k_n\}_n$ a sequence of positive integers such that $k_1<k_2<\cdots$. If there exists $F\in\Delta_+$ such that $F_n^{\cup \hspace{-.52em}\lor k_n}\xrightarrow{w} F$, then $F_n^{\Box \hspace{-.55em}\lor k_n}\xrightarrow{w} {\bf B}_1^M(F)$ as $n\rightarrow \infty$.
\end{theorem}
\begin{proof} 
Denote by $\mathcal{C}(F)$ the set of all continuous points of $F$. Assume that $F_n^{\cup \hspace{-.52em}\lor k_n}\xrightarrow{w} F$ as $n\rightarrow\infty$. Since $\mathbf{B}_1^M(F_n)^{\Box \hspace{-.55em}\lor k_n}=\mathbf{B}_1^M(F_n^{\cup\hspace{-.52em}\lor k_n})\xrightarrow{w} \mathbf{B}_1^M(F)$ as $n\rightarrow\infty$ by our assumption and Proposition \ref{prop:B_1}, it suffices to show that
\begin{align}\label{key_thm7}
\lim_{n\rightarrow\infty}\mathbf{B}_1^M(F_n)^{\Box \hspace{-.55em}\lor k_n}(x)=\lim_{n\rightarrow\infty}F_n^{\Box \hspace{-.55em}\lor k_n}(x), \qquad x\in \mathcal{C}(F).
\end{align}
Consider $x\in \mathcal{C}(F)\cap \{F>0\}$. Our assumption implies that there is $\epsilon_x>0$ such that
\begin{align*}
\frac{F_n(x)}{k_n-(k_n-1)F_n(x)}>\epsilon_x, \hspace{1mm} \text{ equivalently, }\hspace{1mm} 
F_n(x)>\frac{\epsilon_xk_n}{\epsilon_x(k_n-1)+1},
\end{align*}
for sufficiently large $n$. Hence $\lim_{n\rightarrow \infty}F_n(x)=1$, and therefore $k_n(1-F_n(x)^{-1})\sim k_n(F_n(x)-1)$ as $n\rightarrow\infty$. Notice that $\lim_{n\rightarrow\infty}k_n(1-F_n(x)^{-1})=1-F(x)^{-1}$ since $\mathbf{B}_1^M(F_n)^{\Box \hspace{-.55em}\lor k_n}\xrightarrow{w}\mathbf{B}_1^M(F)$ as $n\rightarrow\infty$. Therefore we get
\begin{align*}
\lim_{n\rightarrow\infty}\mathbf{B}_1^M(F_n)^{\Box \hspace{-.55em}\lor k_n}(x)&=\lim_{n\rightarrow\infty} (k_n(1-F_n(x)^{-1})+1)\\
&=\lim_{n\rightarrow\infty} (k_n(F_n(x)-1)+1) = \lim_{n\rightarrow\infty}F_n^{\Box \hspace{-.55em}\lor k_n}(x).
\end{align*}
Next, we consider $x\in \mathcal{C}(F)\cap \{F=0\}$. It is clear that $\lim_{n\rightarrow\infty}\mathbf{B}_1^M(F_n)^{\Box \hspace{-.55em}\lor k_n}(x) = 0$ from our assumption. It also follows from our assumption that
\begin{align*}
\frac{F_n(x)}{k_n-(k_n-1)F_n(x)}<\frac{1}{2},
\end{align*}
for sufficiently large $n$. Therefore we have $F_n(x)<\frac{k_n}{k_n+1}$. This means that $k_nF_n(x)-(k_n-1) <\frac{1}{k_n+1}$. Hence $F_n^{\Box\hspace{-.55em}\lor k_n}(x)=(k_nF_n(x)-(k_n-1))_+\rightarrow 0$ as $n\rightarrow \infty$. 

Finally, we obtain the equation \eqref{key_thm7}, and therefore $F^{\Box \hspace{-.55em}\lor k_n}\xrightarrow{w} \mathbf{B}_1^M(F)$ as $n\rightarrow\infty$.
\end{proof}

\begin{remark}
We can obtain the above theorem even if we change $\Delta_+$ to $\Delta_+^{(0)}$. In this case, we can interpret that the map $\mathbf{B}_1^M$ makes a limit theorem for free regular max-infinitely divisible distributions from one of Boolean max-infinitely divisible distributions.
\end{remark}

\begin{remark}
In Section 6, we prove an equivalence of limit theorems for classical and Boolean max-infinitely divisible distributions by using the map $\mathcal{X}$. Moreover, \cite{BD} has already proved an implication from the classical max-limit theorem to the free max-limit theorem by using the map $\Lambda^{\lor}$. Combining these facts gives another proof of Theorem \ref{thm:limit}.
\end{remark}

Next we show the converse claim of Theorem \ref{thm:limit} under the special case.

\begin{theorem}\label{prop:converse}
Let $\{F_n\}_n$ be a sequence in $\Delta_+$ and $\{k_n\}_n$ a sequence of positive integers such that $k_1<k_2<\cdots$. If there exists $F\in\Delta_+$ with $F>0$ on $[0,\infty)$ such that $F_n^{\Box\hspace{-.55em}\lor k_n}\xrightarrow{w} F$, then $F_n^{\cup \hspace{-.52em}\lor k_n}\xrightarrow{w} G$ as $n\rightarrow \infty$, where
\begin{align*}
G(x):=\begin{cases}
\frac{1}{2-F(x)}, & x\ge 0\\
0, & x<0.
\end{cases}
\end{align*}
\end{theorem}
\begin{proof}
The claim is clear in the case of $x<0$. For any $x\in \mathcal{C}(F)\cap \{F>0\}=\mathcal{C}(F) \cap [0,\infty)$, there is $\epsilon_x>0$ such that $k_nF_n(x)-(k_n-1)>\epsilon_x$ for sufficiently large $n$, and therefore $\lim_{n\rightarrow\infty}F_n(x)= 1$. Then we have 
\begin{align*}
\lim_{n\rightarrow\infty} F_n^{\cup \hspace{-.52em}\lor k_n}(x)&=\lim_{n\rightarrow\infty}\frac{F_n(x)}{1+F_n(x)-(k_nF_n(x)-(k_n-1))}= \frac{1}{2-F(x)}.
\end{align*}
Hence $F_n^{\cup \hspace{-.52em}\lor k_n}\xrightarrow{w} G$ as $n\rightarrow \infty$.
\end{proof}

We give two remarks for Theorem \ref{thm:limit} and Theorem \ref{prop:converse}.

\begin{remark}\label{rem: free-Boolean}
(1) In the setting of Theorem \ref{thm:limit}, note that the convergence of $F_n^{\Box \hspace{-.55em}\lor k_n}$ to $\mathbf{B}_1^M(F)$ (for some $F\in\Delta_+$) does not necessarily imply the convergence of $F_n^{\cup \hspace{-.50em}\lor k_n}$ to $F$. 

We give an example as follows. Consider $\alpha>0$. Let $F_n$ be the following distribution function:
\begin{align*}
F_n(x):=\begin{cases}
\mathbf{P}_\alpha(n^{-1/\alpha}x), & x\ge 1\\
0, & x<1.
\end{cases}
\end{align*}
Then $F_n^{\Box \hspace{-.55em}\lor n}\xrightarrow{w}\mathbf{P}_\alpha=\mathbf{B}_1^M(\mathbf{D}_\alpha)$ as $n\rightarrow\infty$. However, 
\begin{align*}
F_n^{\cup \hspace{-.50em}\lor n}(x)\xrightarrow{n\rightarrow\infty} &\begin{cases}
\mathbf{D}_\alpha(x), & x\ge 1\\
0, & x<1
\end{cases} \neq \mathbf{D}_\alpha(x).
\end{align*}

(2) In Theorem \ref{prop:converse}, we assumed that $F>0$ on $[0,\infty)$ to prove the converse claim of Theorem \ref{thm:limit}. However we expect that the converse claim of Theorem \ref{thm:limit} holds without the special assumption of $F$. 

For example, we consider $F\in\Delta_+$ with $\alpha(F)>0$. Define the following sequences of distribution functions in $\Delta_+$:
\begin{align*}
F_{1,n}(x):&=\begin{cases}
1-\frac{1}{n}+\frac{F}{n}, & x\ge \alpha(F)\\
\frac{1}{\alpha(F)} \left(1-\frac{1}{n}\right)x, & 0\le x <\alpha(F),\\
0, & x<0,
\end{cases}\\
F_{2,n}(x):&=\begin{cases}
1-\frac{1}{n}+\frac{F}{n}, & x\ge \alpha(F)\\
1-\frac{1}{n}, & 0\le x <\alpha(F),\\
0, & x<0,
\end{cases}
\end{align*}
Then $F_{i,n}^{\Box \hspace{-.55em}\lor n}\xrightarrow{w} F$ as $n\rightarrow\infty$ for $i=1,2,3$. Moreover, we have that 
\begin{align*}
F_{1,n}^{\cup \hspace{-.50em}\lor n} &\xrightarrow{w}\begin{cases}
\frac{1}{2-F}, & x\ge \alpha(F)\\
0, &  x< \alpha(F),
\end{cases}\\
F_{2,n}^{\cup \hspace{-.50em}\lor n} &\xrightarrow{w}\begin{cases}
\frac{1}{2-F}, & x\ge \alpha(F)\\
\frac{1}{2}, & 0\le x< \alpha(F),\\
0, & x<0,
\end{cases}
\end{align*}
as $n\rightarrow\infty$. Therefore the convergence of $F_{i,n}^{\Box \hspace{-.55em} \lor n}$ to $F$ ($i=1,2$) implies the convergences of $F_{i,n}^{\cup \hspace{-.50em} \lor n}$ to some distribution function.

According to the above three examples, for any $x\in \mathcal{C}(F) \cap \{F=0\}$, the convergence of the sequence of $F_{n}^{\cup \hspace{-.50em}\lor k_n}(x)$ depends on a situation of $F_n$.\end{remark}

We formulate the conjecture in Remark \ref{rem: free-Boolean} (2) as follows.

\begin{conjecture}
Let $\{F_n\}_n$ be a sequence in $\Delta_+$ and $\{k_n\}_n$ a sequence of positive integers such that $k_1<k_2<\cdots$. If there exists $F\in\Delta_+$ such that $F_n^{\Box\hspace{-.55em}\lor k_n} \xrightarrow{w} F$, then $F_n^{\cup\hspace{-.52em}\lor k_n}\xrightarrow{w} G$ as $n\rightarrow \infty$ for some $G\in \Delta_+$. In particular, we have $G=\frac{1}{2-F}$ on $\{F>0\}$.
\end{conjecture}

\section{Limit theorems for classical and Boolean max-infinitely divisible distributions}
\label{sec:13}

In this section, we give a relation of limit theorems for classical and Boolean max-infinitely divisible distributions. According to \cite[Theorem 4.1]{VV18} (or Section 4.3), we have $\mathcal{X}(\mathbf{D}_\alpha)=\Phi_\alpha$ and $\mathcal{X}^{-1}(\Phi_\alpha)=\mathbf{D}_\alpha$ for any $\alpha>0$, where $\mathcal{X}$ and $\mathcal{X}^{-1}$ were defined in Section 4.1.

The isomorphism $\mathcal{X}$ is an important key to discuss in this section. We give the following theorem.

\begin{theorem}\label{thm:boolean-classical}
Consider a sequence $\{F_n\}_n$ in $\Delta_+$ and a sequence $\{k_n\}_n$ of positive integers with $k_1<k_2<\cdots$. The following conditions are equivalent.
\begin{enumerate}[(1)]
\item There exists $F\in\Delta_+$ such that $F_n^{ \lor k_n}\xrightarrow{w} F$ as $n\rightarrow \infty$;
\item There exists $G\in\Delta_+$ such that $F_n^{\cup \hspace{-.52em}\lor k_n}\xrightarrow{w} G$ as $n\rightarrow\infty$.
\end{enumerate}
If either (1) or (2) holds, then we have $\mathcal{X}(G)=F$ and $\mathcal{X}^{-1}(F)=G$.
\end{theorem}
\begin{proof}
Firstly we show the implication (1) $\Rightarrow $ (2). Assume that $F_n^{\lor k_n}\xrightarrow{w} F$ as $n\rightarrow \infty$. Since $\mathcal{X}^{-1}(F_n)^{\cup \hspace{-.52em} \lor k_n}=\mathcal{X}^{-1}(F_n^{k_n})\xrightarrow{w} \mathcal{X}^{-1}(F)$ as $n\rightarrow\infty$, it suffices to show that 
\begin{align}\label{key_thm9_1}
\lim_{n\rightarrow\infty }\mathcal{X}^{-1}(F_n)^{\cup \hspace{-.50em} \lor k_n}(x)=\lim_{n\rightarrow\infty}F_n^{\cup \hspace{-.50em} \lor k_n}(x) \qquad  x\in \mathcal{C}(F).
\end{align}
Suppose $x\in\mathcal{C}(F)\cap \{F>0\}$. Then $F_n(x)>0$ for sufficiently large $n$ by the assumption (1). Moreover, $F_n(x)\rightarrow 1$ as $n\rightarrow \infty$. Hence we have $k_n(1-F_n^{-1}(x))\sim k_n\log F_n(x)$ as $n\rightarrow\infty$. Notice that $\lim_{n\rightarrow\infty}k_n\log F_n(x)=\lim_{n\rightarrow\infty}\log F_n(x)^{k_n}= \log F(x)$ from the assumption (1). This implies that
\begin{align*}
\lim_{n\rightarrow\infty}\mathcal{X}^{-1}(F_n)^{\cup \hspace{-.50em} \lor k_n}(x)&=\lim_{n\rightarrow\infty} \frac{1}{1-k_n\log F_n(x)}\\
&=\lim_{n\rightarrow\infty} \frac{1}{1-k_n(1-F_n(x)^{-1})}=\lim_{n\rightarrow\infty} F_n^{\cup \hspace{-.50em} \lor k_n}(x).
\end{align*}
Suppose that $x\in\mathcal{C}(F)\cap \{F=0\}$. It is clear that $\lim_{n\rightarrow\infty}\mathcal{X}^{-1}(F_n)^{\cup \hspace{-.50em} \lor k_n}(x)=0$ by the assumption (1). Assume that $F_n^{\cup \hspace{-.50em} \lor k_n} (x)$ does not converge to $0$ as $n\rightarrow\infty$. If necessary, passing to a subsequence, we can find $\delta>0$ such that
\begin{align*}
F_n^{\cup \hspace{-.50em} \lor k_n}(x)=\frac{F_n(x)}{k_n-(k_n-1)F_n(x)}>\delta,
\hspace{1mm} \text{ equivalently, } \hspace{1mm}
F_n(x)>\frac{\delta k_n}{(1-\delta)+\delta k_n},
\end{align*}
for sufficiently large $n$. Therefore we have
\begin{align*}
F_n^{\lor k_n}(x)&>\left\{ \frac{\delta k_n}{(1-\delta)+\delta k_n}\right\}^{k_n}=\left\{1+\left(\frac{1-\delta}{\delta}\right)\frac{1}{k_n}\right\}^{-k_n}\xrightarrow{n\rightarrow\infty} e^{-\frac{1-\delta}{\delta}}>0.
\end{align*}
This is a contradiction for that $\lim_{n\rightarrow\infty }F_n^{\lor k_n}(x)=0$. Hence we have $\lim_{n\rightarrow\infty}F_n^{\cup \hspace{-.50em} \lor k_n}(x)=0$.
Finally, we obtain the equation \eqref{key_thm9_1}, and therefore $F_n^{\cup \hspace{-.52em} \lor k_n}\xrightarrow{w} \mathcal{X}^{-1}(F)$ as $n\rightarrow\infty$.

Next we show the implication (2) $\Rightarrow$ (1). Assume that $F_n^{\cup \hspace{-.52em}\lor k_n}\xrightarrow{w} G$ as $n\rightarrow\infty$. Since $\mathcal{X}(F_n)^{\lor k_n}=\mathcal{X}(F_n^{\cup \hspace{-.52em}\lor k_n})\xrightarrow{w} \mathcal{X}(G)$, it suffices to show that
\begin{align}\label{key_thm9_2}
\lim_{n\rightarrow\infty}\mathcal{X}(F_n)^{\lor k_n}(x) =\lim_{n\rightarrow\infty} F_n^{\lor k_n}(x), \qquad x\in \mathcal{C}(G).
\end{align}
Consider $x\in \mathcal{C}(G)\cap \{G>0\}$.  In the proof of Theorem \ref{thm:limit}, we know that $F_n(x)>0$ for sufficiently large $n$ and $\lim_{n\rightarrow\infty}F_n(x)=1$. It is easy to see that $\lim_{n\rightarrow\infty}k_n(1-F_n(x)^{-1})=1-G(x)^{-1}$ from the assumption (2). Since $k_n(1-F_n^{-1}(x))\sim k_n\log F_n(x)$ for sufficiently large $n$, we obtain
\begin{align*}
\lim_{n\rightarrow\infty}\mathcal{X}(F_n)^{\lor k_n}(x)&=\lim_{n\rightarrow\infty}\exp(k_n(1-F_n(x)^{-1}))\\
&=\lim_{n\rightarrow\infty}\exp (k_n\log F_n(x))=\lim_{n\rightarrow\infty}F_n^{\lor k_n}(x).
\end{align*}
Suppose that $x\in \mathcal{C}(F)\cap \{G=0\}$. We obtain $\lim_{n\rightarrow\infty}\mathcal{X}(F_n)^{\lor k_n}= 0$ from the assumption (2). Moreover it also follows that for arbitrary $\epsilon>0$, we have
\begin{align*}
F_n^{\cup \hspace{-.50em} \lor k_n}(x)=\frac{F_n(x)}{k_n-(k_n-1)F_n(x)}<\epsilon, \hspace{1mm}\text{ equivalently, }\hspace{1mm}
F_n(x)<\frac{k_n\epsilon}{(1-\epsilon)+\epsilon k_n},
\end{align*}
and therefore $F_n^{\lor k_n}(x) \le e^{-\frac{1-\epsilon}{\epsilon}}$ for sufficiently large $n$. Since $\epsilon>0$ is arbitrary, we have $\lim_{n\rightarrow\infty}F_n^{\lor k_n}(x)= 0$. Finally, we get the equation \eqref{key_thm9_2}, and therefore $F_n^{\lor k_n}\xrightarrow{w} \mathcal{X}(G)$ as $n\rightarrow\infty$.

Thus, the equivalence of two conditions holds and $\mathcal{X}(G)=F$ and $\mathcal{X}^{-1}(F)=G$.
\end{proof}

According to the above consideration, it is appropriate that the map $\mathcal{X}$ is said the {\it Boolean-classical max-Bercovici-Pata bijection}. 

Applying discussions in this section and previous section (Section 5), we obtain a relation of limit theorems for classical and freely max-infinitely divisible distributions. However, the following corollary has already been proved in \cite{BD}.

\begin{corollary}
Consider a sequence $\{F_n\}_n$ in $\Delta_+$ and a sequence $\{k_n\}_n$ of positive integers with $k_1<k_2<\cdots$. If there exists $F\in \Delta_+$ such that $F_n^{\lor k_n}\xrightarrow{w} F$, then $F_n^{\Box \hspace{-.55em}\lor k_n}\xrightarrow{w} \Lambda^{\lor}(F)$ as $n\rightarrow\infty$.
\end{corollary}

\begin{proof}
Assume that $F_n^{\lor k_n}\xrightarrow{w} F$ as $n\rightarrow \infty$. By Theorem \ref{thm:boolean-classical}, the assumption is equivalent to $F_n^{\cup \hspace{-.52em} \lor k_n}\xrightarrow{w} \mathcal{X}^{-1}(F)$ as $n\rightarrow\infty$. Moreover, by Proposition \ref{prop:B_1} and Theorem \ref{thm:limit}, the above condition implies that $F_n^{\Box\hspace{-.55em}\lor k_n}\xrightarrow{w} \mathbf{B}_1^M(\mathcal{X}^{-1}(F))= \Lambda^{\lor}\circ\mathcal{X}(\mathcal{X}^{-1}(F))=\Lambda^{\lor}(F)$ as $n\rightarrow\infty$.
\end{proof}

Moreover, the converse claim of the above corollary holds under the special case.

\begin{corollary}\label{cor:free-classical}
Consider a sequence $\{F_n\}_n$ in $\Delta_+$ and a sequence $\{k_n\}_n$ of positive integers with $k_1<k_2<\cdots$. If there exists $F\in \Delta_+$ with $F>0$ on $[0,\infty)$ such that $F_n^{\Box \hspace{-.55em}\lor k_n}\xrightarrow{w} F$, then $F_n^{\lor k_n}\xrightarrow{w} \Pi_{1,F}^{\lor}$ as $n\rightarrow\infty$.
\end{corollary}
\begin{proof}
Assume that $F_n^{\Box \hspace{-.55em}\lor k_n}\xrightarrow{w} F$ as $n\rightarrow\infty$, where $F\in\Delta_+$ with $F>0$ on $[0,\infty)$. By Theorem \ref{prop:converse}, we have $F_n^{\cup \hspace{-.52em}\lor k_n}\xrightarrow{w} G \in\Delta_+$ as $n\rightarrow\infty$, where $G(x)=\frac{1}{2-F(x)}$ for $x\in[0,\infty)$ and $G(x)=0$ for $x\in (-\infty,0)$. By Theorem \ref{thm:boolean-classical}, for we have $F_n^{\lor k_n}\xrightarrow{w} \mathcal{X}\left(G\right)=\Pi_{1,F}^{\lor}$ as $n\rightarrow\infty$.
\end{proof}

For example, if $F=\Pi_{\lambda, G}^{\Box \hspace{-.55em}\lor}$ for some $0\le\lambda<1$ and a distribution function $G$ on $\mathbb{R}$ in Corollary \ref{cor:free-classical}, then we get $\Pi_{1,F}^{\lor}=\Pi_{\lambda,G}^{\lor}$. 

From the same reason of Remark \ref{rem: free-Boolean} (1), in general, the convergence of $F_n^{\Box \hspace{-.55em} \lor k_n}$ to $\Lambda^{\lor}(F)$ (for some $F\in\Delta_+$) does not necessarily imply the convergence of $F_n^{\lor k_n}$ to $F$.

\vspace{0.6cm}
\hspace{-5.5mm}{\it Yuki Ueda\\
Department of Mathematics, Hokkaido University,\\
Kita 10, Nishi 8, Kita-Ku, Sapporo, Hokkaido, 060-0810, Japan\\
email: yuuki1114@math.sci.hokudai.ac.jp}



\begin{thebibliography}{}
%
%

\bibitem{Ando}
T.\ Ando, Majorization, doubly stochastic matrices and comparison of eigenvalues, Linear Algebra Appl. {\bf 118} 163-248 (1989).

\bibitem{BR77} 
A.\ A.\ Balkema, S.\ I.\ Resnick, Max-Infinite Divisibility,  J.\ Appl.\ Probability {\bf 14}, no.\ 2, 309-319 (1977).

\bibitem{AV06} 
G.\ Ben Arous, D.\ V.\ Voiculescu, Free extreme values,  Ann.\ Probab.\ {\bf 34}, no.\ 5, 2037-2059 (2006). 

\bibitem{BD} 
F.\ Benaych-Georges, T.\ Cabanal-Duvillard, A Matrix Interpolation between Classical and Free Max Operations. I. The Univariate Case, J.\ Theoret.\ Probab.\ {\bf 23}, no.\ 2, 447-465 (2010).

\bibitem{BN08} 
S.\ T.\ Belinschi, A.\ Nica, On a remarkable semigroup of homomorphisms with respect to free multiplicative convolution, Indiana Univ.\ Math.\ J.\ {\bf 57}, no.\ 4, 1679-1713 (2008). 

\bibitem{BP99} 
H.\ Bercovici, V.\ Pata, Stable laws and domains of attraction in free probability theory (with an appendix by P. Biane), Ann.\ of Math.\ (2) {\bf 149}, no.\ 3, 1023-1060 (1999).

\bibitem{BP00}
H.\ Bercovici, V.\ Pata, A free analogue of Hin\u{c}in's characterization of infinite divisibility, Proc.\ Amer.\ Math.\ Soc.\ {\bf 128}, no. 4, 1011-1015 (2000).

\bibitem{BV93} 
H.\ Bercovici, D.\ V.\ Voiculescu, Free convolution of measures with unbounded support, Indiana Univ.\ Math.\ J.\ {\bf 42}, no.\ 3, 733-773 (1993). 


\bibitem{CH18} 
S.\ Chakraborty, R.\ S.\ Hazra, Boolean convolutions and regular variation, ALEA, Lat.\ Am.\ J.\ Probab.\ Math.\ Stat.\ {\bf 15}, 961-991 (2018).

\bibitem{HM13} 
R.\ S.\ Hazra, K.\ Maulik, Free subexponentiality, Ann.\ of Probab.\ {\bf 41}, no.\ 2, 961-988 (2013).

\bibitem{HW18}
H-W.\ Huang, J-C.\ Wang, Bi-Free Extreme Values, arXiv:1811.10007.

\bibitem{K51}
R.\ V.\ Kadison,  Order properties of bounded self-adjoint operators. Proc.\ Amer.\ Math.\ Soc.\ {\bf 2}, 505-510 (1951).

\bibitem{NS06} 
A.\ Nica, R.\ Speicher, Lectures on the Combinatorics of Free Probability, xvi+417pp, London Math.\ Soc.\ Lecture Note Series, 335, Cambridge University Press, Cambridge (2006).

\bibitem{O71}
M.\ P.\ Olson, The selfadjoint operators of a von Neumann algebra form a conditionally complete lattice. Proc. Amer. Math. Soc., {\bf 28}, 537-544 (1971).

\bibitem{SW} 
R.\ Speicher, R. Woroudi, Boolean convolution, Fields Inst.\ Commun.\ {\bf 12}, 267-279 (1997).


\bibitem{VV18} 
J.\ G.\ Vargas, D.\ V.\ Voiculescu, Boolean Extremes and Dagum Distributions, arXiv:1711.06227.

\end{thebibliography}
\end{document}